\documentclass[12pt]{amsart}

\usepackage{amsfonts,amsmath,latexsym,amssymb,verbatim,amsbsy,amsthm}
\usepackage{mathrsfs}
\usepackage{nicefrac}
\usepackage{graphicx}
\usepackage{nicefrac}
\usepackage{dsfont}
\usepackage{mathrsfs}
\usepackage{cancel}
\usepackage[normalem]{ulem}
\usepackage{color}
\usepackage{pifont}

\usepackage{array}
\usepackage{slashbox}

\usepackage{setspace} 
\usepackage{caption,subcaption} 

\usepackage{algorithm} 

\usepackage{algpseudocode} 

\usepackage{grffile}  

\usepackage[top=1in, bottom=1in, left=1in, right=1in]{geometry}

\usepackage[dvipsnames]{xcolor}
\usepackage{relsize}

\usepackage[colorlinks=true, pdfstartview=FitV, linkcolor=RoyalBlue,citecolor=ForestGreen, urlcolor=blue]{hyperref}

\numberwithin{equation}{section}
\numberwithin{figure}{section}
\numberwithin{table}{section}
\usepackage{float}

\renewcommand{\geq}{\geqslant}

\renewcommand{\leq}{\leqslant}

\newcommand{\ds}{\displaystyle} 
\newcommand{\be}{\begin{equation}}
\newcommand{\ee}{\end{equation}}

\theoremstyle{plain}
\newtheorem{THEOREM}{Theorem}[section]

\newtheorem{theorem}[THEOREM]{Theorem}
\newtheorem{corollary}[THEOREM]{Corollary}

\newtheorem{lemma}[THEOREM]{Lemma}
\newtheorem{proposition}[THEOREM]{Proposition}

\theoremstyle{definition}

\newtheorem{definition}[THEOREM]{Definition}

\theoremstyle{remark}
\theoremstyle{question}

\newtheorem{remark}[THEOREM]{Remark}

\newtheorem{example}[THEOREM]{Example}

\newcommand{\R}{\mathbb{R}}

\newcommand{\calC}{{\mathcal C}}
\newcommand{\calP}{{\mathcal P}}
\newcommand{\bbC}{{\mathbb C}}
\newcommand{\bbD}{{\mathbb D}}
\newcommand{\bbE}{{\mathbb E}}
\newcommand{\bbF}{{\mathbb F}}
\newcommand{\bbH}{{\mathbb H}}
\newcommand{\bbI}{{\mathbb I}}
\newcommand{\bbJ}{{\mathbb J}}
\newcommand{\bbL}{{\mathbb L}}
\newcommand{\bbN}{{\mathbb N}}
\newcommand{\bbR}{{\mathbb R}}
\newcommand{\bbT}{{\mathbb T}}
\newcommand{\bbU}{{\mathbb U}}

\newcommand{\bx}{\mathbf{x}}

\newcommand{\bk}{{\mathbf k}}

\newcommand{\bF}{{\mathbf F}}

\newcommand{\by}{\mathbf{y}}

\newcommand{\rd}{\textnormal{d}}

\newcommand{\bu}{\mathbf{u}}
\newcommand{\bw}{\mathbf{w}}

\newcommand{\nicehf}{\nicefrac{1}{2}}
\newcommand{\vnorm}[1]{[\![#1]\!]}

\newcommand{\ddt}{\frac{\textnormal{d}}{\textnormal{d}t}}

\newcommand{\Real}{Re\hspace*{0.08cm}}
\newcommand{\LN}{{\bbL}_{{}_N}}
\newcommand{\DbbL}{\Delta t\hspace*{0.04cm}\bbL}
\newcommand{\DLN}{\Delta t\hspace*{0.04cm}\LN}

\newcommand{\DpN}{\bbD^+_{N}}
\newcommand{\DDpN}{\Delta t\hspace*{-0.04cm}\cdot \hspace*{-0.04cm}a\hspace*{0.04cm}\DpN} 
\newcommand{\DoN}{\bbD^0_{N}}
\newcommand{\DDoN}{\Delta t\hspace*{-0.04cm}\cdot \hspace*{-0.04cm}a\hspace*{0.04cm}\DoN}
\newcommand{\DFN}{\bbD^{\bbF}_{N}}

\newcommand{\EN}{\bbE_{{}_N}}
\newcommand{\diffQ}{Q} 
\newcommand{\QEN}{\diffQ(\EN)}
\newcommand{\DQEN}{\Delta t\hspace*{-0.04cm}\cdot \hspace*{-0.04cm}\QEN}
\newcommand{\diffq}{q}  

\newcommand{\HN}{{\bbH}_{{}_N}}
\newcommand{\calH}{{\mathcal H}}
\newcommand{\calHN}{\calH_{{}_N}}
\newcommand{\TN}{{\bbT}_N}
\newcommand{\calT}{{\mathcal T}}
\newcommand{\calTN}{\calT_{{}_N}}
\newcommand{\UN}{{\bbU}_{{}_N}}
\newcommand{\RN}{{\bbR}_{{}_N}}
\newcommand{\Ps}{{\calP}_s}

\newcommand{\sA}{\mathscr{A}}
\newcommand{\parm}{q}
\newcommand{\mesh}{\delta}  
\newcommand{\const}{K} 

\newcommand\bigzero{\makebox(0,0){\text{\huge0}}}

\newcommand{\CFL}{\calC} 

\begin{document}

\title[Runge-Kutta methods are stable]{Runge-Kutta methods are stable}

\author{Eitan Tadmor}
\address{Department of Mathematics and Institute for Physical Science \& Technology\newline \hspace*{0.3cm}   University of Maryland, College Park}
\email{tadmor@umd.edu}

\date{\today}

\subjclass{Primary, 65M10; Secondary, 35A12}

\keywords{L2-stability, resolvent condition,  Runge-Kutta methods, region
of absolute stability, numerical range,  spectral sets, finite-difference schemes, method of lines.}

\thanks{\textbf{Acknowledgment.} Research was supported in part by ONR grant N00014-2112773.}

\begin{abstract}
We prove that Runge-Kutta (RK) methods for numerical integration of  arbitrarily large systems of Ordinary Differential Equations are linearly stable. 
Standard stability arguments --- based on spectral analysis, resolvent condition or strong stability,  fail to secure the stability of  arbitrarily large RK systems. We explain the  failure of   different approaches, offer a new stability theory and demonstrate a few examples.
\end{abstract}

\maketitle
\setcounter{tocdepth}{1}
\tableofcontents

\section{Introduction --- the quest for stability}
Runge-Kutta (RK) methods are the methods of choice for numerical integration  of systems of Ordinary Differential Equations (ODEs). In particular, such methods are used routinely for integration of large systems of ODEs encountered in various applications, for example ---  in molecular dynamics in Chemistry, in many particle systems in Physics, in phase field dynamics in Materials Science, and in spatial discretization of time-dependent PDEs of increasingly large dimension, so-called ``method of lines''. The stability of RK methods encoded in terms of their \emph{region of absolute stability} is well documented, \cite{HNW1993,Ise1996,But2008}. We therefore devote this Introduction to clarify the claim alluded to in the title. 

We consider systems of ODEs, 
\[
\dot{\by}=\bF(t,\by),
\]
 which govern an $N$-vector of unknown solution, $\by(t) \in \R^N$, subject to prescribed initial data, $\by(t_0)=\by_0$. As a canonical example for one of the most widely used numerical integrators  we mention the 4-stage RK method,  which computes an approximate solution, $\{\bu_n=\bu(t_n)\}_{n>0}$, at successive time steps $t_{n+1}:=t_n+\Delta t$, \cite[\S II.1]{HNW1993}, 
\be\label{eq:nonlinear-RK4}
\bu_{n+1}=\bu_n+ \frac{\Delta t}{6}\big(\bk_1+2\bk_2+2\bk_3+\bk_4\big), \qquad 
\left\{\begin{array}{l}
\bk_1=\bF(t_n,\bu_n)\\
\displaystyle \bk_2=\bF\big(t_{n+\nicehf}, \bu_n+(\nicefrac{\Delta t}{2})\bk_1\big)\\
\displaystyle \bk_3=\bF\big(t_{n+\nicehf}, \bu_n+(\nicefrac{\Delta t}{2})\bk_2\big)\\
\bk_4=\bF\big(t_{n+1}, \bu_n+\Delta t \bk_3\big).
\end{array}\right.
\ee
The \emph{linearized stability analysis}  examines the behavior of  \eqref{eq:nonlinear-RK4} for linear systems, $\bF(t,\by)=\LN\by$, 
 \be\label{eq:linear}
 \dot{\by}=\LN\by,
 \ee
 where \eqref{eq:nonlinear-RK4} is reduced to
 \be\label{eq:RK4}
 \bu_{n+1}=\left({\mathbb I}+\DLN+\frac{1}{2}(\DLN)^2+\frac{1}{6}(\DLN)^3 + \frac{1}{24}(\DLN)^4\right)\bu_n, \quad n=0,1,\ldots.\tag{RK4}
 \ee
The  corresponding iterations for a general $s$-stage \emph{explicit} RK method  take the form 
\be\label{eq:RKs}
\bu_{n+1}=\Ps(\DLN)\bu_n, \quad n=0,1,2,\ldots, \qquad \Ps(z):=\sum_{k=0}^s a_kz^k, \ \ a_k \in \bbR, \ a_s\neq 0.
\ee
Different $\{a_k\}_{k=0}^s$ dictate different RK methods with emphasize on different aspects of accuracy, efficiency and stability.
The resulting  $s$-stage RK methods, \eqref{eq:RKs}, involve  $N\times N$ matrices, denoted $\LN$ to highlight the fact that they are parameterized with respect to $N$. As already noted above, such large matrices are often encountered in applications, notably in contemporary problems which involve high-dimensional data sets/neural networks, e.g., \cite{HR2017,E2017, CRBD2018,Mis2018}.  We therefore pay particular attention to the question  of RK stability that is uniform with respect  to the increasingly large dimension $N$.

 Following \cite[\S2]{Tad2002}, we consider \eqref{eq:linear} for the class of \emph{semi-bounded} $\LN$\!\!'s, namely --- $\LN$\!\!'s for which there exist  constants $\eta,\const_{\bbH}>0$ independent of $N$, and \emph{uniformly} positive-definite symmetrizers, $\HN$\!\!'s,  such that\footnote{Throughout the paper, we use  $\const_{\square}$ to denote different constants which are independent of $N$.},
 \[
 \HN\LN^\top+\LN\HN \leq 2\eta\HN, \qquad 0<\const^{-1}_{\bbH}\leq \HN\leq \const_{\bbH}.
 \]
 It follows that the solutions of the corresponding semi-bounded  ODEs \eqref{eq:linear} subject to arbitrary initial data $\by(0)=\by_0$, satisfy
 \[
 |\by(t)|_{\ell^2} \leq \const_{\bbH}e^{\eta t}|\by_0|_{\ell^2}.
 \]
Replacing $\LN$ with $\LN\!-\eta\bbI$, allows us to consider without loss of generality the case $\eta=0$, corresponding to \emph{negative definite} $\LN$\!\!'s, 
 \be\label{eq:negative}
 \HN\LN^\top+\LN\HN \leq 0, \qquad 0<\const^{-1}_{\bbH}\leq \HN\leq \const_{\bbH}.
 \ee
Solutions of ODE governed by such negative\footnote{Throughout the work, we use the term  `negative' for short of `negative definite'.} $\LN$\!\!'s satisfy subject to arbitrary initial data $\by_0$, 
 \begin{equation}\label{eq:stable}
 |\by(t)|_{{}_{\ell^2}} \leq \const_{\bbH}|\by_0|_{{}_{\ell^2}}.
 \end{equation}
{\bf Stability of RK scheme}. The notion of stability of  RK schemes requires the numerical  solution to satisfy the bound corresponding to \eqref{eq:stable}. To this end, one is focused on a family of negative $\LN$\!\!'s parametrized by their dimension $N$. The  $s$-stage  RK scheme \eqref{eq:RKs} is \emph{stable}, if  there exist  constants,  $\const_{\bbL}>0$ and  $\CFL_s>0$ independent of $N$, such that   solutions  of \eqref{eq:RKs} subject to arbitrary  initial data $\bu_0$ satisfy, for small enough time step $\Delta t$,
\be\label{eq:stable-RK-scheme}
\textnormal{Stability of RK scheme}: \qquad 
|\bu_n|_{{}_{\ell^2}}\leq \const_{\bbL}|\bu_0|_{{}_{\ell^2}}, \qquad n=0,1,2,\ldots.
\ee
The restriction of having small enough time step is encoded in terms of the bound 
\be\label{eq:CFL}
\Delta t \cdot \|\LN\|\leq \CFL_s;
\ee
in the context of method of lines, the time-step restriction is related to the celebrated Courant-Friedrichs-Levy (CFL) condition, \cite{CFL1928}, and we shall therefore often refer to the time-step restriction \eqref{eq:CFL} as a CFL condition.\newline
The notion of  stability encoded in \eqref{eq:stable-RK-scheme} amounts to the question of power-boundedness of $\Ps(\DLN)$,
\be\label{eq:power-boundedness}
\|\Ps^n(\DLN)\|\leq \const_{\bbL}, \qquad n=0,1,2,\ldots.
\ee
\begin{remark}[{\bf Stability and linearization}] The general notion of stability for semi-bounded $\LN$\!\!'s, limits the exponential stability bound to a finite time interval,
\[
|\bu_n|_{{}_{\ell^2}}\leq \const_{\bbL}e^{\eta t}|\bu_0|_{{}_{\ell^2}},\qquad n\cdot \Delta t\leq t.
\]
 Since we restrict attention to negative $\LN$\!\!'s, we may as well let $n\in {\mathbb N}$. This notion of stability is invariant against low-order perturbations,  \cite{Kre1962},\cite{Str1964}, and therefore allows  to recover the stability  of RK schemes for smooth solutions of fully nonlinear problems, 
$\dot{\by}=\bF(t,\by)$. To this end, one  can linearize and freeze coefficients at arbitrary $t=t_*$, 
arriving at the linearized system \eqref{eq:linear}, 
\[
\dot{\by}=\LN\by \ \ \textnormal{with} \ \  
 \LN=\frac{\partial \bF\big(t_*,\by(t_*)\big)}{\partial \by}.
 \]
 We shall not dwell on the details, expect for referring to our discussion on stability in presence of variable coefficients in section \ref{sec:periodic-var} below. This motivates our focus on the  question of  
 linearized stability, where $\LN$ is a substitute for the $N\times N$ gradient matrix freezed at arbitrary state.
\end{remark}

 \subsection{Spectral stability analysis}
The standard approach  to address the question of power-boundedness is  \emph{spectral analysis},  in which \eqref{eq:power-boundedness} requires  $\displaystyle \mathop{\max}_{1\leq k\leq N}|\lambda_k\big(\Ps(\DLN)\big)|\leq 1$. By the spectral mapping theorem, 
\be\label{eq:spectral-mapping}
\lambda_k\big(\Ps(\DLN)=\Ps\big(\Delta t \lambda_k(\LN)\big),
\ee
 which  leads to the necessary stability condition, requiring  small enough time-step dictated by the \emph{region of absolute stability} associated with \eqref{eq:RKs},
\be\label{eq:absolute-stability}
\Delta t \cdot\lambda_k(\LN) \in \sA_s, \quad k=1,2,\ldots, N, \qquad \sA_s:=\{z\in {\mathbb C}\ : \ |\Ps(z)|\leq 1\}.
\ee
Conversely, consider the favorite scenario in which $\LN$ is diagonalizable, 
\[
\TN\LN\TN^{-1}=\Lambda, \qquad \Lambda=\left[\begin{array}{ccccc}
\lambda_1(\LN)& 0 & \ldots & \ldots &  0\\
0 & \hspace*{-0.4cm} \lambda_2(\LN) &\ddots &  & \vdots\\
\vdots &  \ddots &  & \ddots & \vdots\\
\vdots & \ddots &   \ddots & \ddots & 0\\
0 & \ldots & \ldots & 0 & \lambda_N(\LN)
\end{array}
\right].
\]
Then $\Ps(\DLN)=\TN^{-1}\Ps(\Delta t\Lambda)\TN$ and \eqref{eq:absolute-stability}  implies 
\be\label{eq:CCbound}
\|\Ps^n(\DLN)\|= \|\TN^{-1}\Ps^n(\Delta t\Lambda)\TN\|\leq 
\|\TN^{-1}\|\cdot \|\TN\|.
\ee
This guarantees the stability of  RK schemes for systems of  finite \emph{fixed} dimension\footnote{The precise necessary and sufficient  characterization  for  power-boundedness of a single matrix, $\|\calP^n\|\leq \const$, requires that the eigenvalues $\lambda_k(\calP)$ are inside the unit disc and those on the unit circle are simple; the constant $\const$ may still depend on the dimension of $\calP$.}. However, here we insist that the stability  sought in  \eqref{eq:stable-RK-scheme}  will apply uniformly for   increasingly large systems,
and since the condition number on the right of \eqref{eq:CCbound}, $\|\TN^{-1}\|\cdot \|\TN\|$, may grow with $N$,  the spectral  condition \eqref{eq:absolute-stability} is not enough to secure the desired   uniform-in-$N$ stability bound. Indeed, as we  elaborate in section \ref{sec:spectral-not-enough} below, the  general question of stability, uniformly in $N$, \underline{cannot} be addressed solely in terms of  spectral analysis.

\subsection{Resolvent stability} We now appeal to a stronger notion of stability of RK method.  An $s$-stage  RK method $\Ps(\cdot)$ is \emph{stable} if the corresponding RK schemes  \eqref{eq:RKs} are stable  for \underline{all} negative $\LN$\!\!'s,
 \be\label{eq:stable-RK-method} \textnormal{Stability of RK method}: \qquad 
 \|\Ps^n(\DLN)\|\leq \const_{\bbL} \quad \textnormal{for all negative} \ {\LN}\!\!\textnormal{'s}.
 \ee
 Observe that we are making a distinction between the stability of RK \emph{scheme} --- which examines the boundedness of RK protocol $\Ps^n(\DLN)$ for a specific family of negative $\LN$\!\!'s,  vs. the stability of RK \emph{method} --- which examines the behavior of RK protocol $\Ps^n(\Delta t\ \cdot)$, for \emph{all} negative $\LN$\!\!'s. 
 
 This stronger notion of stability restricts the class of stable RK methods. In particular, 
 their stability question should apply to the scalar ODEs,
$\dot{y}=\lambda y$, for \underline{all} negative $\Real \lambda \leq 0$, which in turn implies that  \eqref{eq:absolute-stability}  must hold for purely imaginary $\lambda=i\sigma$, so that \mbox{$|\Ps(i\Delta t \sigma)|\leq 1$}, for small enough step-size, $\Delta t$. In other words, a stable RK method \emph{must} satisfy the following interval condition.
\begin{definition}[{\bf Imaginary Interval condition}\footnote{So-called ``local stability along the imaginary line'' in \cite[Definition 2.1]{KS1992}.}]\label{def:interval-condition} A Runge-Kutta method is said to satisfy the \emph{imaginary interval condition} if there exists a constant $R_s>0$ such that
\be\label{eq:along-iaxis}
|\Ps(i\sigma)|\leq 1, \qquad -R_s\leq \sigma \leq R_s.
\ee
\end{definition}
In other words, the region of absolute stability of a stable RK method must contains a non-trivial  interval along the imaginary axis
$[-iR_s, iR_s]\subset  \sA_s$.
This secures  the stability of RK method for scalar hyperbolic ODEs, $\dot{y}=i\sigma y$, with small enough step-size $\Delta t \sigma <R_s$.

\smallskip\noindent
The interval condition  excludes the standard 1-stage forward Euler method 
 (for historical perspective of Euler's method which dates back to 1768 see \cite[\S1]{Wan2010}), 
\be\label{eq:FE}\tag{RK1}
 \textnormal{Forward Euler}: \hspace{1cm} \bu_{n+1}=\left(\bbI+\DLN\right)\bu_n,
 \ee
 for which $\calP_1(z)=1+z \ \leadsto \ |\calP_1(i\sigma)|>1$ for all $\sigma\neq 0$. The imaginary interval condition \eqref{eq:along-iaxis} also excludes the 2-stage Heun's method \cite[\S 8.3.3]{DB1974} (also known as modified Euler method),
\be\label{eq:RK2}\tag{RK2}
\textnormal{Heun method}: \quad \bu_{n+1}=\left(\bbI+\DLN+\frac{1}{2}(\DLN)^2\right)\bu_n,  
\ee
since $\calP_2(z)=1+z+\frac{1}{2}z^2 \ \leadsto \ |\calP_2(i\sigma)|>1$ for all $\sigma\neq 0$.\newline
On the other hand, the 3-stage Kutta method, 
\be\label{eq:RK3}
\textnormal{Kutta method}: \quad \bu_{n+1}=\left(\bbI+\DLN+\frac{1}{2}(\DLN)^2+\frac{1}{6}(\DLN)^3\right)\bu_n,\tag{RK3}
\ee
as well as the 4-stage Runge-Kutta method, \eqref{eq:RK4}, and its higher-order embedded version RK45 of Dormand-Prince method \cite{DP1980,KS1992,HNW1993}, do satisfy the interval condition with $R_3=\sqrt{3}$, and respectively, $R_4=2\sqrt{2}$; this is depicted  in figure \ref{fig:absRK14}. A precise characterization of    general $s$-stage RK methods satisfying the interval condition was given in \cite[Theorem 3.1]{KS1992} and will be recalled in \eqref{eq:interval-condition} below. 
\begin{figure}
\includegraphics[scale = 0.62]{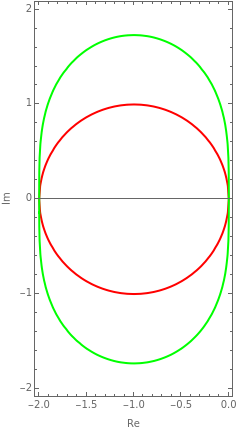} \hspace*{0.5cm}
\includegraphics[scale = 0.62]{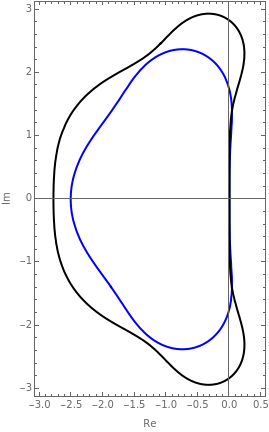}
\caption{Regions of absolute stability, $\sA_s, \ s=1,2$ (left) and $s=3,4$ (right)}\label{fig:absRK14}
\end{figure}

\noindent
The interval condition \ref{def:interval-condition} is necessary for stability of a RK method.  Kreiss \& Wu, \cite{KW1993}, proved the converse  in the sense  that the interval condition is sufficient for \emph{resolvent stability}; namely --- \eqref{eq:along-iaxis} implies that there exists  constants $K_R>0$ and $0<\CFL_s<R_s$, independent of $N$, such that  for small step-size, 
\be\label{eq:resolvent-stability}
\|\big(z\bbI-\Ps(\DLN)\big)^{-1}\| \leq \frac{\const_R}{|z|-1}, \qquad \forall|z|>1, \qquad \Delta t\cdot\|\LN\|\leq \CFL_s.
\ee
Resolvent stability guarantees the stability of RK schemes for systems of finite \emph{fixed} dimension,  in view of the Kreiss matrix theorem, \cite{Kre1962}, \cite[\S4.9]{RM1967}. Indeed, in \cite{Tad1981} and its improvement \cite{LT1984}, it was proved that \eqref{eq:resolvent-stability} implies
\be\label{eq:resolventN}
\|\Ps^n(\DLN)\|\leq 2e\const_R N, \qquad n=1,2,\ldots.
\ee
However, as we shall elaborate  in section \ref{sec:resolvent-not-enough} below,  the $N$-dependent bound on the right cannot be completely removed and  hence resolvent stability   does \underline{not} secure  the desired stability uniformly in growing $N$.

\subsection{Strong stability} A Runge-Kutta  scheme \eqref{eq:RKs} 
is \emph{strongly stable} if there exists $\const_\calT>0$ independent of $N$ such that
$\Ps(\DLN)$ is uniformly similar to a contraction,
\be\label{eq:contract}
	\|\calTN\Ps(\DLN)\calTN^{-1}\|\leq 1, \qquad \|\calTN^{-1}\|\cdot\|\calTN\|\leq \const_{\calT}.
\ee
A strongly stable RK scheme is clearly stable, for
\be\label{eq:sub-multiplicativity}
\|\Ps^n(\DLN)\|=\|\calTN^{-1}\big(\calTN\Ps(\DLN)\calTN^{-1}\big)^n\calTN\|
\leq \|\calTN^{-1}\|\cdot\|\calTN\|\leq \const_{\calT}
\ee
The choice $\calTN=\bbT_N$ recovers \eqref{eq:CCbound}  as a special case of 
\eqref{eq:sub-multiplicativity}.\newline
To secure strong stability it remains to construct a uniformly bounded symmetrizer $\calHN:=\calTN^*\calTN$ with $\ds 0< {\const^{-1}_\calT} \leq \calHN \leq \const_\calT$. 
We addressed  this issue in \cite{Tad2002}, proving the strong stability of the 3-stage RK method \eqref{eq:RK3} with symmetrizer $\calHN=\HN$ and $\CFL_3=1$, thus providing the first example of a RK method which is stable uniformly for arbitrarily large system of ODEs. It was later extended to all  $s$-stage RK methods of  order  $s=3[\textnormal{mod}4]$, \cite{SS2019}.
What about the other $s$-stage RK methods --- does this argument of  strong stability can be extended using proper symmetrizers, $\calHN$,  for arbitrary $s$?  In \cite{Tad2002} we conjectured that the 4-stage \eqref{eq:RK4} fails  strong stability in the sense  that it is not uniformly similar to a contraction, or equivalently ---  as outlined in section \ref{sec:sstability-not-enough} below, that there exist no symmetrizer $\calHN:=\calTN^*\calTN$ such that  \eqref{eq:contract}${}_{s=4}$ holds.
This was confirmed in \cite[Proposition 1.1]{SS2017} and was later extended in \cite[Theorem 2]{AAJ2023}, where it was shown that strong stability  \emph{fails} for  all $s$-stage RK methods with $s\in4\bbN$. 

\smallskip\noindent
{\bf The stability question for RK schemes}. We come out from the above discussion, lacking a definitive answer to the  question of stability of RK schemes/methods for arbitrarily large systems of ODEs. Thus, for example, the stability question for the widely used RK4  remains open. At this stage, the three different approaches --- spectral analysis, resolvent condition and strong stability  failed  to  determine whether   RK4  for example, is stable uniformly in $N$. We  therefore raise the question:

\begin{quote}
 Are the Runge-Kutta methods \eqref{eq:RKs},\eqref{eq:stable-RK-method} stable for arbitrarily large systems? 
 \end{quote}
 
 \noindent
The title of the paper is an affirmative answer to this question. The answer is  given in section \ref{sec:spectral-set} in terms of the numerical range of $\LN$.

\section{Spectral, resolvent and strong stability analysis are not enough}
In this section, we further elaborate with specific counterexamples,  on the failure of spectral analysis, resolvent condition and strong stability to capture the uniform-in-$N$ stability of general $s$-stage RK schemes/methods. Spectral and resolvent analysis are shown to be too weak to secure stability, while strong stability argument is too restrictive.

\subsection{Spectral analysis is not enough}\label{sec:spectral-not-enough}
We recall the spectral analysis led to the necessary stability condition   \eqref{eq:absolute-stability} 
\[
\Delta t\cdot\lambda_k(\LN) \subset \sA_s, \qquad  k=1,2,,\ldots,N. 
\]
As noted above, this  spectral  condition is not sufficient to secure stability in case of ill-conditioned eigensystems, $\|\TN^{-1}\|\cdot\|\TN\|$, which grows with $N$. An alternative approach, trying to circumvent this difficulty  of ill-condtioning is to use a unitary triangulation 
\[
\Ps(\DLN)=\UN^*\big(\Lambda_{{}_{\calP}} + \RN\big)\UN, \qquad \Lambda_{{}_{\calP}}:= \Ps(\Delta t\Lambda),
\]
 where $\Lambda$ and $\Lambda_{{}_{\calP}}$ are the diagonals made of the eigenvalues of  $\LN$ and, respectively, $\Ps(\DLN)$,  and $\RN$ is a nilpotent  upper triangular matrix, $(\RN)_{ij}=0, \ j\leq i$. Since  $\|\Ps^n(\DLN)\|= \|(\Lambda_{{}_{\calP}} + \RN)^n\|$,
 it remains to study the power-boundedness of the triangular matrix $\Lambda_{{}_{\calP}}+\RN$.
 But we claim that even  a most  favorable scenario, in which the spectral stability analysis \eqref{eq:absolute-stability} secures the eignevalues   \emph{strictly} inside the unit disc, 
 \be\label{eq:favorable}
 \theta:=\mathop{\max}_{1\leq k\leq N}|\Ps\big(\Delta t \lambda_k(\LN)\big)|<1,
 \ee
  will not suffice to guarantee the stability of RK method. 
 Indeed,  we may assume without restriction that $\RN$ is arbitrarily small by  its further re-scaling\footnote{$\|\cdot\|_F$ refers to Frobenius norm,
 $\|A\|^2_F=\textnormal{trace}(A^\top A)$}, so that
 \[
 \|S_\delta \RN S^{-1}_\delta\|=\|\{\bbR_{ij}\delta_\epsilon^{i-j}\}_{j>i}\|\leq \epsilon, \qquad S_\delta=\left[\begin{array}{cccc}
\delta_\epsilon& 0 & \ldots & 0\\
0 & \delta_\epsilon^2 &\ddots & \vdots\\
\vdots & \ddots & \ddots & \vdots\\
0 & \ldots & \ldots & \delta_\epsilon^N
\end{array}
\right], \quad \delta_\epsilon:=\frac{\epsilon}{\|\RN\|_F}.
 \]
 Here, an arbitrary $\epsilon>0$ is at our disposal to be determined below. 
 It follows that
 \[
 \|\Ps^n(\DLN)\|= \|\UN^*S^{-1}_\delta\big(\Lambda_{{}_{\calP}} + S_\delta\RN S^{-1}_\delta\big)^n S_\delta\UN\|\leq \|S^{-1}_\delta\|\times (\|\Lambda_{{}_{\calP}}\|+\epsilon)^n \times \|S_\delta\|.
 \]
 By assumption, $\|\Lambda_{{}_{\calP}}\|=\theta<1$. Set $\epsilon:=\nicefrac{1}{2}(1-\theta)$, we then end up with the stability bound
 \be\label{eq:atneqN}
  \|\Ps^n(\DLN)\| \leq \delta^{1-N}_\epsilon  \left(\frac{1+\theta}{2}\right)^n = \Big(\frac{2\|\RN\|_F}{1-\theta}\Big)^{N-1}\Big(\frac{1+\theta}{2}\Big)^n.
 \ee
 This bound secures the stability of finite dimensional systems -- in fact, it recovers the well-known fact that matrices of finite \emph{fixed} dimension with eigenvalues strictly inside the unit disc have exponentially decreasing iterates. But the argument breaks down when we examine the dependence  on $N$, since the bound \eqref{eq:atneqN} is \emph{not} uniform in $N$: for $n=N-1$, for example,  we find that unless $\RN$ is sufficiently small\footnote{To avoid an exponential growth of the upper-bound in 
 \eqref{eq:atneqN} requires $\|\RN\|_F \leq \frac{1-\theta}{1+\theta}$; a more delicate tuning of the scaling parameter $\delta_\epsilon$ shows that uniform bound is achieved for $\|\RN\|_F < 1-\theta$.}, then there is an \emph{exponential growth} in $N$,
 \be\label{eq:grow-with-N}
 \|\Ps^n(\DLN)\| \leq \Big(\frac{2\|\RN\|_F}{1-\theta}\Big)^{N-1}\Big(\frac{1+\theta}{2}\Big)^n_{\big| n=N-1} = \left(\|\RN\|_F\frac{1+\theta}{1-\theta}\right)^{N-1}.
 \ee
 This bound is sharp in the sense that the  power-growth hinted on the right of \eqref{eq:grow-with-N} is realized by the powers of the increasingly large $N\times N$ Jordan blocks
 \be\label{eq:Jordan}
 \|\bbJ^n_\parm\| \sim \left(\frac{2}{1-\parm}\right)^N \left(\frac{1+\parm}{2}\right)^n, \qquad \bbJ_\parm:=\left[\begin{array}{ccccc}
 -\parm & 1+\parm & \ldots & \ldots & 0\\
0 & -\parm & 1+\parm & \ddots & \vdots\\
\vdots & \ddots & \ddots &  \ddots & \vdots\\
0 & \ldots & \ddots &  -\parm & 1+\parm\\
0 & \ldots & \ldots & \ldots & -\parm
 \end{array}\right].
 \ee
 Although  $|\lambda_k(\bbJ_\parm)|<1$ for $-1<\parm<1$, there is a nonuniform  
 growth of $\|\bbJ^n_\parm\|$ with $0<\parm<1$, corresponding to $\parm=\theta$ in \eqref{eq:grow-with-N}, when $n\sim N \uparrow \infty$. These increasingly large Jordan blocks realize the extreme case of ill-conditioning warned in \eqref{eq:CCbound}. 
 
\subsubsection{Instability of forward Euler scheme}\label{sec:unstable-FE}The extremal example  \eqref{eq:Jordan} is not just of academic interest.  The following classical example,   \cite[\S6.6]{RM1967},\cite[\S3]{KW1993},\cite[\S5.1]{Tad2002} sheds light on what can go wrong with spectral analysis. Consider the transport equation with fixed speed $a>0$
\be\label{eq:yt=ayx}
\left\{\ \begin{split}
y_t(x,t)&=ay_x(x,t), \quad (t,x)\in \bbR_+\times (0,1)\\
y(1,t) & =0.
\end{split}\right.
\ee
Its spatial part is discretized using one-sided spatial differences on equi-spaced grid, $\{x_\nu:=\nu \Delta x\}_{\nu=0}^N, \, \Delta x=\nicefrac{1}{N}$, covering the interval $[0,1]$,
\be\label{eq:one-sided-semi-discrete}
\left\{\begin{split}
\ddt y(x_\nu,t)&=  a\, \frac{y(x_{\nu+1},t)-y(x_\nu,t)}{\Delta x}, \qquad  \nu=0, 1, \ldots, N-1, \\
\quad y(x_N,t)&=0.
\end{split}\right.
\ee
This amounts to method of lines  for the $N$-vector of unknowns, $\by(t):=\big(y(x_0,t),\ldots, y(x_{N-1},t)\big)^\top$, governed by the $N\times N$ semi-discrete system in terms of the forward-difference operator $\DpN$,
\be\label{eq:transport}
\dot{\by}(t)=a\DpN\by, \qquad \DpN:=\frac{1}{\Delta x}\left[\begin{array}{ccccc}
 -1& 1 & \ldots & \ldots & 0\\
0 & -1 & 1 & \ddots & \vdots\\
\vdots & \ddots & \ddots &  \ddots & \vdots\\
0 & \ldots & \ddots &  -1 & 1\\
0 & \ldots & \ldots & \ldots & -1
 \end{array}\right].
\ee
Observe that $\DpN$ is semi-bounded --- in fact it is \emph{strictly dissipative} in the sense that
\[
(\DpN)^\top+\DpN \leq -2\Big(1-\cos\big(\frac{\pi}{N+1}\big)\Big)\bbI_{N\times N}.
\] 
This system \eqref{eq:transport} is integrated using one-stage Forward Euler method, \eqref{eq:FE}, augmented with boundary condition $u(x_N,t)=0$,
\be\label{eq:FE-Jordan}
\bu_{n+1}=\calP_1(\DDpN)\bu_n,\quad \bu_n:=\big(u(x_0,t^n),\ldots, u(x_{N-1},t^n)\big)^\top, \qquad n=0,1,2,\ldots,
\ee
which encodes  the fully discrete finite difference scheme
\be\label{eq:one-sided}
\left\{\begin{split}
 \frac{u(x_\nu,t^{n+1})-u(x_\nu,t^n)}{\Delta t}&=  a \frac{u(x_{\nu+1},t^n)-u(x_\nu,t^n)}{\Delta x}, \qquad  \nu=0, 1, \ldots, N-1,\\
  u(x_N,t^{n+1})&=0.
  \end{split}\right.
\ee
The computation proceeds  with hyperbolic scaling of fixed mesh ratio, $\Delta t/\Delta x$. This is precisely  the regime $N\sim n$ indicated in \eqref{eq:grow-with-N}, in which case it is known that the forward Euler scheme \eqref{eq:one-sided}  is \emph{unstable}, if it violates the CFL condition  $0<{a\Delta t}/{\Delta x} <1$. Observe that $\calP_1(\DDpN)$ amounts to a Jordan block,
\[
\calP_1(\DDpN)={\mathbb I}+\DDpN=\bbJ_\parm, \quad \parm=a\mesh-1, \quad \mesh:=\frac{\Delta t}{\Delta x}.
\]
 Therefore,  the instability of  $\bbJ_\parm$ with $\parm\in (0,1]$ follows, corresponding to  $1<a\mesh <2$, which was already claimed by the bound \eqref{eq:Jordan}. In particular, the RK1  scheme \eqref{eq:FE-Jordan} is unstable, despite having $|\lambda_k(\calP_1(\DDpN)|= |\parm| < 1$.

Now  consider integration of \eqref{eq:transport} using 4-stage \eqref{eq:RK4}. Spectral stability analysis
\[
|\lambda_k(\calP_4(\DDpN)|=|\calP_4(-a\mesh)|\leq 1, 
\]
leads to the  CFL condition, $ 0< a\mesh \leq R_4=2\sqrt{2}$,  which \emph{fails} to guarantee stability, since it does not take capture the power-growth of the increasingly large Jordan block $a\mesh\DpN$. We conclude that even in the most favorable scenario \eqref{eq:favorable}, spectral analysis  is not enough to secure a uniform-in-$N$ stability of RK methods  for increasingly large systems.

\subsection{Resolvent stability is not enough}\label{sec:resolvent-not-enough} 
Recall that the  imaginary interval condition \eqref{eq:along-iaxis} is necessary for the stability of RK method.  Kreiss and Wu \cite[Theorem 3.6]{KW1993} proved that the converse holds in the sense of \emph{resolvent stability}. Here,  resolvent stability  is interpreted in  the sense that  there exists a constant $K_R>0$ independent of $N$, such that  for all negative $\LN$'s, if the time step is small enough, $\Delta t\cdot\|\LN\|\leq \CFL_s$,  then the corresponding $s$-stage RK method satisfies 
\be\label{eq:resolvent-stable}
\|\big(z{\mathbb I}-\Ps(\DLN)\big)^{-1}\|\leq \frac{K_R}{|z|-1}, \qquad \forall |z|>1.
\ee
The size of the time step is dictated by region of absolute stability, $\sA_s$, specifically ---  $\CFL_s\leq R_s$ is the radius of largest half disc inscribed inside $\sA_s$, 
\[
B^-_{\CFL_s}(0):=\large\{z :\, \Real z <0, \ |z|<\CFL_s\large\} \subset \sA_s, \qquad   
 \sA_s=\big\{z\in{\mathbb C} : \ |\Ps(z)|\leq 1\big\}.
 \]
The notion of stability in the sense of power-boundedness, \eqref{eq:power-boundedness}, implies that the resolvent condition holds with $K_R=\const_{\bbL}$.
The Kreiss Matrix Theorem, \cite{Kre1962},\cite[\S4.9]{RM1967}, states that the converse holds for a family of matrices with a \emph{fixed} dimension. 
Yet this does not enable us to conclude the uniform-in-$N$ power-boundedness stability of RK method sought in \eqref{eq:stable-RK-method}, since the resolvent bound \eqref{eq:resolvent-stable} may still allow growth \mbox{$\|\Ps^n(\DLN)\|\lesssim N\const_R$}.
In \cite{Tad1981} we conjectured that this  linear dependence on $N$  is the best possible. This was confirmed  in \cite{LT1984} proving that
\[
\sup_{A\in M_N(\bbC)}\frac{\sup_{|z|>1} (|z|-1)\|(z\bbI-A)^{-1}\|}{\sup_{n\geq 1} \|A^n\|} \sim eN.
\]
The above linear-growth-in-$N$ behavior was exhibited by a sequence of increasingly large $N\times N$ Jordan blocks, $A_N=N\bbJ_0$. We observe that the $A_N$'s in this case are not resolvent bounded uniformly in $N$ ; it is   only the ratio on the left that exhibits the sharp linear bound in $N$.  
A concrete example of a  family of matrices in $M_N(\bbC)$ which are resolvent stable yet their powers  admit logarithmic growth in $N$ was constructed in \cite{MS1965}. 

 \begin{remark}[{\bf Dissipative resolvent condition}] In \cite{Tad1986} we considered a stronger  resolvent condition of the form
\be\label{eq:resolventD}
\|\big(z\bbI- \Ps(\DLN)\big)^{-1}\|\leq \frac{\const_R}{|z-1|}, \quad \forall \{z\,:\, |z|\geq 1, z\neq 1\}.
\ee
In \cite{Rit1953} it was  proved that \eqref{eq:resolventD} implies $n^{-1}\|\Ps^n\|\stackrel{n\rightarrow \infty}{\longrightarrow}0$. In \cite{Tad1986} we stated the improved logarithmic bound $\|\Ps^n(\DLN)\|\lesssim \log(n)$; this  was proved in \cite{Vit2004a}. More on the dissipative resolvent \eqref{eq:resolventD} and related notions of stability  can be found in \cite{Vit2004b,Vit2005,Sch2016}.   The dissipative resolvent bound \eqref{eq:resolventD} reflects a flavour of coercivity condition which will be visited in section \ref{sec:coercivity} below;  however,  it does \underline{not} secure uniform-in-$N$ power-boundedness. A more precise notion of  a dissipative resolvent condition of order $2r>0$ requires the existence of $\eta_r>0$ such that
\be\label{eq:resolventDs}
\|(z\bbI- \Ps(\DLN))^{-1}\|\leq \frac{\const_R}{\textnormal{dist}\{z, \Omega_r\}}, \quad \forall z\notin \Omega_r:=\{w \, : \, |w|+\eta_r|w-1|^{2r}\leq 1\}.
\ee
 The resolvent bound \eqref{eq:resolventDs} reflects the classical notion of  ``dissipativity of order $2r$'' due to Kreiss \cite{Kre1964}. It remains an open question whether \eqref{eq:resolventDs} implies uniform-in-$N$ power-boundedness.
\end{remark} 
\subsection{Strong stability is not enough}\label{sec:sstability-not-enough}
The contractivity  stated in \eqref{eq:contract}, \mbox{$\|\calTN\Ps(\DLN)\calTN^{-1}\|\leq 1$} with uniformly bounded  $\|\calTN^{-1}\|\cdot\|\calTN\|\leq \const_{\calT}$, is equivalent to strong stability in the sense that there exist
 uniformly positive definite symmetrizer $\calHN$ and  $\const_\calH>0$, such that
\be\label{eq:strongstability}
\Ps^\top(\DLN)\calHN\Ps(\DLN) \leq \calHN, \qquad 0<\frac{1}{\const_\calH} \leq \calHN \leq \const_\calH.
\ee
Just set $\calHN=\calTN^*\calTN$ with uniformly bounded $\const_\calH=\const_\calT$. In other words, \eqref{eq:strongstability} tells us that\footnote{We let $|\cdot|_\calH$ denote the weighted norm, $|\bw|^2_\calH=\langle \bw,\calH\bw\rangle$, and $\|\cdot\|_\calH$ denote the corresponding induced matrix norm, $\|\calP\|_\calH:=\max_{\bw\neq 0}|\calP\bw|_\calH/|\bw|_\calH$.}
\be\label{eq:ssPs}
\|\Ps(\DLN)\|_{\calHN}\leq 1, \qquad \Delta t\cdot\|\LN\|\leq \CFL_s.
\ee
This coincides with the usual notion of strong stability,\footnote{also called monotonicity in the literature on Runge-Kutta methods.} e.g., \cite{Tad2002,Ran2021}.
It follows that  a strongly stable RK scheme, $\bu_{n+1}=\Ps(\DLN)\bu_n$, satisfies 
\[
|\bu_{n+1}|_{\calHN} = |\Ps(\DLN)\bu_n|_{\calHN} \leq |\bu_n|_{\calHN} \leq \ldots \leq |\bu_0|_{\calHN},
\]
 and hence the RK iterations satisfy the uniform-in-$N$ stability bound,
$\displaystyle |\bu(t_n)|_{{}_{\ell^2}} \leq \const_\calH|\bu_0|_{{}_{\ell^2}}$.
The strong stability of   the 3-stage RK method \eqref{eq:RK3} with symmetrizer $\calHN=\HN$ and $\CFL_3=1$, was proved in \cite{Tad2002} and  was later extended in \cite[Theorem 4.2]{SS2019} to all  $s$-stage RK methods of  order  $s=3[\textnormal{mod}4]$, namely --- for small enough time step, $\Delta t\cdot\|\LN\|\leq \CFL_s$, there holds, 
\be\label{eq:SSPs}
\|\Ps(\DLN)\|_{\HN} \leq 1, \qquad \Ps(z)=\sum_{k=0}^s\frac{z^k}{k!}, \quad s=3[\textnormal{mod}4].
\ee
As mentioned above, this line of arguing stability by construction of the strong stability symmetrizer,  fails to extend to $s$-stage RK methods with $s\in 4\bbN$, \cite{SS2017,RO2018,AAJ2023}. But this does not mean that the latter RK methods are necessarily unstable. 
Indeed, the  general question whether  stable methods are necessarily strongly stable was addressed in \cite{Fog1964} --- they are not.  It  leaves  open the possibility that the question  stability can be pursued by other approaches --- other than strong stability. This  will be addressed in the next section. 

\section{Numerical range and stability of coercive Runge-Kutta schemes}\label{sec:range}
\subsection{Numerical range}\label{sec:numerical-range}
 We let $\ell^2_H(\bbC^N)$ denote the weighted Euclidean space associated with a given  positive definite  matrix $H>0$, and equipped with
\[
\langle \bx,\by\rangle_H:=  \bx^*H\by, \quad |\bx|_H^2:=\langle\bx,H\bx\rangle, \quad H>0.
\]
Let $A\in M_N(\bbC)$ be an $N\times N$ matrix with possibly complex-valued entries.
The $H$-weighted \emph{numerical range}, $W_H(A)$, is the  set in the complex plane
\[
W_H(A):=\big\{\langle A\bx,\bx\rangle_H\ : \ \bx\in \bbC^N, \ |\bx|_H=1\big\}.
\]
In the  case of the standard Euclidean framework corresponding to $H=\bbI$,  we drop the subscript $\bbH=\bbI$  and remain with the usual $|\cdot|^2_{{\ell^2}} =\langle\cdot,\cdot\rangle$, and the corresponding numerical range denoted $W(A)$.
If $A$ is real symmetric then $W(A)$ is an interval  on the real line (and conversely --- if $W(A)$ is a real interval then $A$ is symmetric, \cite[Problem 3.9]{Kat1995}); if $A$ is skew-symmetric then $W(A)$ is an interval on the imaginary line. 
For general $A$'s, the Hausdorff-Toeplitz theorem asserts that $W(A)$ is convex set in ${\mathbb C}$. As an example, we compute the numerical range of the $N\times N$ translation matrix, $J_0$, 
\be\label{eq:J0}
\bbJ_0:= \left[\begin{array}{ccccc}
0& 1 & \ldots & \ldots & 0\\
0 & 0 & 1 & \ddots & \vdots\\
\vdots & \ddots & \ddots &  \ddots & \vdots\\
0 & \ldots & \ddots &  0 & 1\\
0 & \ldots & \ldots & \ldots & 0\end{array}\right]_{N\times N}.
\ee
 For any unit vector 
$\bx=(x_1,x_2,\ldots,x_N)^\top$ we set a new unit vector 
$x_j(\xi):=e^{ij\xi}x_j$ to find
\[
\langle \bbJ_0\bx(\xi),\bx(\xi)\rangle = \sum_{j=1}^{N-1}x_{j+1}(\xi)\overline{x_{j}(\xi)}=e^{i\xi}\langle \bbJ_0\bx,\bx\rangle, \qquad x_j(\xi):=e^{ij\xi}x_j,
\]
which proves  that $W(\bbJ_0)$ is a disc centered at the origin, $B_{\rho}(0)$; its radius, $\rho=\rho_{{}_N}$, is found by considering the eigenvalues   $\lambda_k(\Real\bbJ_0)= \cos(\frac{k\pi}{N+1}), \, k=1,2,\ldots, N$: since for any $A$, ${\Real}W(A)=W({\Real}A)$, we find 
$\rho_{{}_N}=\lambda_1({\Real}\bbJ_0)=\cos(\frac{\pi}{N+1})$, and we conclude that $W(\bbJ)$ is the disc $B_{\rho_{{}_N}}(0)$,
\be\label{eq:W-of-J}
W(\bbJ_0)=\{z \, : \, |z|\leq \rho_{{}_N}\}, \qquad \rho_{{}_N}=\cos\big(\frac{\pi}{N+1}\big).
\ee

\subsection{The nuemrical radius} 
The numerical radius of $A\in M_N(\bbC)$ is given by 
\[
r_{{}_{\!H}}(A):=\max_{|\bx|_H=1}|\langle A\bx,\bx\rangle_H|.
\]
The role of the numerical radius in addressing the question of stability was pioneered in the celebrated work of Lax \& Wendroff, \cite{LW1964}, in which they proved the stability of their 2D Lax-Wendroff scheme, i.e., power-boundedness of a family amplification matrices, $\|G^n\|\leq Const.$,  by securing  $r(G)\leq 1$. The original proof, by induction on $N$ (!), was later replaced by Halmos inequality, \cite{Hal1967},\cite{Pea1966}
\be\label{eq:Halmos}
r(G^n) \leq r^n(G).
\ee
Note  that although the numerical radius is not sub-multiplicative, that is --- although  $r(AB) \leq r(A)r(B)$   may fail for general $A,B\in M_N({\mathbb R})$, \cite{GT1982},  Halmos' inequality   states that  it holds whenever $A=B$.\newline
Since for all $A$'s there holds $\|A\|\leq 2r(A)$, \eqref{eq:Halmos} immediately yields the stability asserted by  Lax \& Wendroff
\be\label{eq:LW-stability}
r(G)\leq 1 \ \leadsto \ \|G^n\|\leq 2,
\ee
 and more important for our purpose --- power-boundedness is secured uniformly in $N$. 
It is straightforward to extend these arguments to the weighted framework, \cite[\S3]{Tad1981}
\be\label{eq:weighted-Halmos}
r_{{}_{\bbH}}(G^n) \leq r_{{}_{\bbH}}^n(G), \ \textnormal{and therefore} \ r_{{}_{\bbH}}(G)\leq 1 \ \leadsto \ \|G^n\|\leq 2\const_\bbH, \quad 0< \const^{-1}_{\bbH}\leq \bbH\leq \const_\bbH.
\ee
\begin{remark} H.-O. Kreiss proved the LW stability \eqref{eq:LW-stability} by linking it to a (strict) resolvent condition
\[
r(A) \leq 1 \ \leadsto    \|(z\bbI- A)^{-1}\|\leq \frac{1}{|z|-1}, \qquad \forall |z|>1
\]
and conversely, \cite{Spi1993}, the numerical range  is the smallest set $S=W(A)$, which induces the strict resolvent condition 
\[
\|(z\bbI-A)^{-1}\| \leq \frac{1}{\textnormal{dist}(z, S)}, \quad \forall z\in S^{c}.
\]
\end{remark}

\subsection{Coercivity and RK stability}\label{sec:coercivity} We turn to verify the stability of the 1-stage  forward Euler scheme \eqref{eq:FE}, 
\[
\bu_{n+1}=(\bbI+\DLN)\bu_n.
\]
There are two regions of interest in the complex plane that we need to consider: the weighted numerical range, $W_{\HN}\!(\LN)$, and the region of absolute stability associated with forward Euler, $\sA_1=\{z\, :\, |1+z|\leq 1\}$. We make the assumption  that the time step $\Delta t$ is small enough so that 
\be\label{eq:CFL-of-FE}
\Delta t\,W_{\HN}\!(\LN) \subset \sA_1, \qquad \sA_1=\{z :\, |1+z|\leq 1\},
\ee
 then   
 \be\label{eq:stability-of-RK1}
 r_{{}_{\HN}}\!\!(\calP_1(\DLN)) =\max_{|\bx|_{\HN}=1}|1+\langle \DLN\bx,\bx\rangle_{\HN}|=
 \max_{z\in \Delta t W_{\HN}\!(\LN)}|1+ z| \leq  \max_{z\in \sA_1}|\calP_1(z)| =1.
 \ee
  We summarize by stating the following. 
  \begin{theorem}[{\bf Numerical range stability of} \ref{eq:FE}]\label{thm:coercive-FE}
  Consider the forward Euler scheme associated with  1-stage forward Euler method \eqref{eq:FE},
 \[
 \bu_{n+1}=(\bbI+\DLN)\bu_n, \quad n=0,1,2,\ldots,
 \]
 with assume the CFL condition \eqref{eq:CFL-of-FE} holds. Then the scheme is stable, and the following stability bound holds 
 \[
 |\bu_n|_{{}_{\ell^2}}\leq 2\const_{\bbH}|\bu_0|_{{}_{\ell^2}}, \qquad \forall n\geq 1.
 \]
\end{theorem}
  \begin{example}\label{exm:one-sided} As an example for theorem \ref{thm:coercive-FE}
  we consider the one-sided differences  \eqref{eq:transport},
\be\label{eq:one-sided-revisited}
\DDpN=a\frac{\Delta t}{\Delta x}\left[\begin{array}{ccccc}
 -1& 1 & \ldots & \ldots & 0\\
0 & -1 & 1 & \ddots & \vdots\\
\vdots & \ddots & \ddots &  \ddots & \vdots\\
0 & \ldots & \ddots &  -1 & 1\\
0 & \ldots & \ldots & \ldots & -1
 \end{array}\right] = a\mesh\big(-{\mathbb I} +\bbJ_0\big),  \quad a>0, \quad\mesh=\frac{\Delta t}{\Delta x}.
\ee
 By translation and dilation, $W(\DDpN)=a\mesh\big(-1\oplus W(\bbJ_0)\big)$, where  \eqref{eq:W-of-J} tells us that  $W(\bbJ_0)$ is the ball $B_{\rho_N}(0)$. Hence $W(\DDpN)$ is given by the shifted ball,
\be\label{eq:W-of-Jq}
W(\DDpN)=\Big\{z \, : \, |z+a\mesh|\leq a \mesh\rho_{{}_N}\Big\}, \qquad \mesh=\frac{\Delta t}{\Delta x}, \quad \rho_{{}_N}=\cos\big(\frac{\pi}{N+1}\big).
\ee
In particular, $W(\DDpN) \subset B_1(-1)$ uniformly in $N$ if and only if the CFL condition $a\mesh\leq 1$ holds, which in turn secures the stability of the 1-stage forward Euler  method, \eqref{eq:FE}, for one-sided the transport equation\eqref{eq:transport}, $\bu_{n+1}=(\bbI+\DDpN)\bu_n$.
\end{example}
  \begin{corollary}[{\bf Stability of forward Euler scheme}]\label{cor:FE}
 Consider the forward Euler scheme \eqref{eq:one-sided} associated with 1-stage RK method \eqref{eq:FE},
 \[
 \bu_{n+1}=\calP_1(\DDpN)\bu_n, \quad n=0,1,2,\ldots.
 \]
 The scheme is stable under the CFL condition, $0<a\mesh\leq 1$, and the following stability bound  holds 
 $ |\bu_n|_{{}_{\ell^2}}\leq 2|\bu_0|_{{}_{\ell^2}}, \ \forall n\geq 1$.
  \end{corollary}
 The last corollary  can be recast in terms of a  stability statement for $\bbJ_\parm=\calP_1(\DDpN)$,
 \[
 \|\bbJ^n_\parm\|\leq 2, \qquad \parm \in (-1,0).
 \]
  This  complements the statement of instability of $\bbJ_\parm$ in the range $\parm\in (0,1]$, discussed in section \ref{sec:unstable-FE}.\newline
We note that  the stability of $\bbJ_\parm, \ \parm \in [-1,0)$ can be independently verified  by  its induced  $\ell^1$-norm , 
 \be\label{eq:stability-ell1}
 \|\bbJ_\parm\|_{\ell^1}= |-\parm|+|1+\parm|= 1 \ \leadsto \ 
\|\bbJ^n_\parm\|_{\ell^1} \leq 1, \qquad \parm \in [-1,0).
\ee
However, the $\ell^2$-stability $\|\bbJ_\parm\|_{\ell^1}\leq 2$ stated in corollary \ref{cor:FE} and the $\ell^1$-stability \eqref{eq:stability-ell1} are \emph{not} equivalent uniformly in $N$.
Also, $\bbJ_\parm$ is subject to  $\ell^2$   von-Neumann stability analysis,
\cite[\S4.7]{RM1967}
\[
\max_\varphi\left|-\parm+(1+\parm)e^{i\varphi}\right|=1, \qquad  \parm \in [-1,0).
\] 
However, since the underlying problem \eqref{eq:one-sided} is not periodic,
  von Neumann stability analysis may not suffice:  it requires the normal mode analysis \cite{Kre1968}
  to prove $\ell^2$-stability. 
Thus, the numerical range   argument summarized in corollary \ref{cor:FE}  offers a genuinely different approach  of addressing the question of stability, at least for 1-stage \ref{eq:FE}.

 \begin{remark}[{\bf Coercivity}]\label{rem:corecivity}
The CFL restriction  encoded in \eqref{eq:CFL-of-FE}, $|\langle \DLN\bx,\bx\rangle_{\HN}+1|\leq 1$, leads to the sub-class of  negative $\LN$\!\!'s   which  satisfy the \emph{coercivity bound}
   \be\label{eq:coercivity}
  2\textnormal{\Real}\langle\LN\bx,\bx\rangle_{\HN} \leq -\beta|\langle\LN\bx,\bx\rangle_{\HN}|^2, \qquad \forall \bx \in \{\bbC^N: \, |\bx|_{\HN}=1\}.
  \ee
  Indeed, if $\LN$ is $\beta$-coercive in the sense that  \eqref{eq:coercivity} holds with $\beta>0$,   then \eqref{eq:CFL-of-FE} is satisfied under the CFL condition $\Delta t\leq \beta$, and stability follows, $r_{{}_{\HN}}\!\!(\bbI+\DLN) \leq 1$.
   We note that \eqref{eq:coercivity} places a weaker coercivity condition than the  stronger notion of coercivity introduced in \cite{LT1998}
  \be\label{eq:coercivity-98}
 \LN^\top\HN+\HN\LN \leq -\beta\LN^\top\HN\LN, \qquad \beta>0.
  \ee
  Indeed, the latter   implies \eqref{eq:coercivity},  for
  \[
   2\textnormal{\Real}\langle\LN\bx,\bx\rangle_{\HN} \leq -\beta\langle \LN^\top\HN\LN\bx,\bx\rangle = -\beta|\LN\bx|^2_{\HN} \leq -\beta|\langle\LN\bx,\bx\rangle_{\HN}|^2, \qquad |\bx|_{\HN}=1.
  \]
  One can then revisit the coercivity-based examples for stable RK methods in \cite{LT1998} using the relaxed coercivity \eqref{eq:coercivity}. The notion of $\beta$-coercivity is related to the dissipative resolvent condition \eqref{eq:resolventD} but we shall not dwell on this point in this work.
   \end{remark}

 \subsection{Numerical range stability of SSP RK\hspace*{-0.02cm}$s$}\label{sec:numerical-RKs}
We extend theorem \ref{thm:coercive-FE}  to multi-stage RK methods using their Strong Stability Preserving (SSP)  format\cite[\S3]{GST2001}. We demonstrate the first three cases of RK\hspace*{-0.02cm}$s$, $s=2,3,4$.\newline
Assume that the numerical range stability \eqref{eq:stability-of-RK1} holds. For example, the CFL condition $\Delta t\leq \beta$ for $\beta$-coercive $\LN$\!'s, \eqref{eq:coercivity}, implies  $r_{{}_{\HN}}\!\!(\bbI+\DLN)\leq 1$. Then, for the 2-stage RK method, \eqref{eq:RK2}, we have by Halmos inequality \eqref{eq:Halmos}
\[
r_{{}_{\HN}}\!\!\big(\calP_2(\DLN)\big) \leq 
\nicefrac{1}{2}+\nicefrac{1}{2}\,r^2_{{}_{\HN}}\!\!(\bbI+\DLN)  \leq 
\nicefrac{1}{2}+\nicefrac{1}{2}=1, \quad \calP_2(\DLN) \equiv\nicefrac{1}{2}\,\bbI+\nicefrac{1}{2}\,(\bbI+\Delta t \LN)^2.
\]
Similarly, the  3-stage RK method \eqref{eq:RK3} can be expressed as 
\[
 \calP_3(\DLN) \equiv  \nicefrac{1}{3}\,\bbI+\nicefrac{1}{2}\,(\bbI+\DLN) + \nicefrac{1}{6}\,(\bbI+\DLN)^3,  
\]
and hence if \eqref{eq:stability-of-RK1} holds, then the stability of \eqref{eq:RK3} follows from Halmos inequality,
\[
r_{{}_{\HN}}\!\!\big(\calP_3(\DLN)\big) \leq 
\nicefrac{1}{3}+\nicefrac{1}{2}\,r_{{}_{\HN}}\!\!(\bbI+\DLN) + \nicefrac{1}{6}\,r^3_{{}_{\HN}}\!\!(\bbI+\DLN) \leq 
\nicefrac{1}{3}+\nicefrac{1}{2}+\nicefrac{1}{6}=1.
\]
A similar argument applies to the 4-stage RK \eqref{eq:RK4}, 
\[
 \calP_4(\DLN) \equiv  \nicefrac{3}{8}\,\bbI+\nicefrac{1}{3}\,(\bbI+\DLN) + \nicefrac{1}{4}\,(\bbI+\DLN)^2 
    + \nicefrac{1}{24}\,(\bbI+\DLN)^4; 
\]
the numerical stability \eqref{eq:stability-of-RK1},
  $r_{{}_{\HN}}\!\!(\bbI+\Delta t \LN)\leq 1$  
implies the  stability of RK4, 
\[
\begin{split}
r_{{}_{\HN}}\!\!\big(\calP_4(\Delta t\LN)\big) &\leq 
 \nicefrac{3}{8}+\nicefrac{1}{3}\,r_{{}_{\HN}}\!\!(\bbI+\DLN) + \nicefrac{1}{4}\,r^2_{{}_{\HN}}\!\!(\bbI+\DLN)^2
    + \nicefrac{1}{24}\,r^4_{{}_{\HN}}\!\!(\bbI+\DLN)\\ 
& \leq \nicefrac{3}{8}+\nicefrac{1}{3}+\nicefrac{1}{4}+\nicefrac{1}{24}=1.
\end{split}
\]
We summarize by stating
\begin{corollary}[{\bf Coercivity implies stability of RK\hspace*{-0.02cm}$s$}, $s=2,3,4$]\label{cor:SSP}
Consider the RK schemes 
\[
\bu_{n+1}=\Ps(\DLN)\bu_n, \ n=0,1,2,\ldots, \quad  s=2,3,4.
\]
 Assume the numerical range stability \eqref{eq:stability-of-RK1} holds. In particular if $\LN$ is $\beta$-coercive in the sense of  \eqref{eq:coercivity}, and that the CFL condition, $\Delta t \leq \beta$, is satisfied. Then these $s$-stage RK schemes are  stable, 
\[
|\bu(t_n)|_{\HN}\leq 2|\bu(0)|_{\HN} \ \leadsto \ |\bu(t_n)|_{\ell^2}\leq 2\const_{\bbH}|\bu(0)|_{\ell^2}.
\]
\end{corollary}
The building block of corollary \ref{cor:SSP} is the condition of  numerical range stability \eqref{eq:stability-of-RK1} originated with \eqref{eq:FE}. While this argument is sharp for the 1-stage forward Euler, this SSP-based argument  is too restrictive for multi-stage RK\hspace*{-0.02cm}$s$. In particular, corollary \ref{cor:SSP} rules out the large sub-class of negative yet non-coercive  $\LN$\!\!'s, due to a numerical range which has non-trivial intersection with the imaginary axes. In particular, this includes the important  sub-class of skew-symmetric (hyperbolic) $\LN$\!\!'s with purely imaginary numerical range. For example, if the one-sided differences in \eqref{eq:transport} are replaced by centered-differences
\be\label{eq:centered}
\bu_{n+1}=(\bbI+\DDoN)\bu_n, \qquad \DoN:=\frac{1}{\Delta x}\left[\begin{array}{ccccc}
 0& 1 & \ldots & \ldots & 0\\
-1 & 0 & 1 & \ddots & \vdots\\
\vdots & \ddots & \ddots &  \ddots & \vdots\\
0 & \ldots & -1 &  0 & 1\\
0 & \ldots & \ldots & -1 & 0
 \end{array}\right]_{N\times N}.
\ee
The numerical range lies on the imaginary interval, $W(\DDoN)= [-iR, iR]$ with $R=R_N=a\mesh\cos(\frac{\pi}{N+1})$. The 1-stage forward Euler \eqref{eq:centered} fails to satisfy the imaginary interval condition, and therefore, corollary \ref{cor:SSP} fails to capture the stability of the corresponding RK\hspace*{-0.02cm}$s$,
$\bu_{n+1}=\Ps(\DDoN)\bu_n$
 for $s=3,4$.

\section{Spectral sets and stability of Runge-Kutta methods}\label{sec:spectral-set}
We now turn our attention to the stability of multi-stage RK methods, $\Ps(\DLN)$. Clearly, spectral analysis is not enough. On the other hand, direct  computation based on $\ell^1$ or $\ell^2$-von Neumann analysis is not accessible:  even the entries in the example of one-sided differences, $\Ps(\DDpN)$, for $s=3,4$, become excessively complicated to write down. Instead, we suggest to pursue a stability argument based on numerical radius along the lines of \eqref{eq:stability-of-RK1}, starting with
\[
r(\Ps(\DLN)) =\max_{\substack{|\bx|=1\\ \bx\in \bbC^N}}\big|\sum_{k=0}^s a_k\langle (\DLN)^k\bx,\bx\rangle\big|.
\]
This requires a proper functional calculus  of numerical range, relating the sets $W(\Ps(\DLN))$ and  $\{|\Ps(z)|\,, : \, z\in W(\DLN)\}$, similar to the role of the spectral mapping theorem \eqref{eq:spectral-mapping} as the centerpiece of spectral stability analysis. To this end we recall the notion of a $K$-\emph{spectral set} developed in \cite{Del1999,CG2019}, which dates back to von Neumann \cite{vN1951}; we refer to \cite{SdV2023} for a most recent overview. 
\begin{definition}[{\bf $K$-spectral sets}] Given $A\in M_N(\bbC)$, we say that $\Omega\subset \bbC$ is a \emph{$K$-spectral set} of $A$ if there exists a finite $K>0$ such that for all analytic $f$'s bounded on $\Omega$, there holds
\be\label{eq:spectral-set}
\|f(A)\|_H \leq K\max_{z\in \Omega}|f(z)|.
\ee
\end{definition}
In a remarkable  work, \cite{Cro2007}, Crouzeix proved that for every matrix $A$, the numerical range $W_H(A)$ is a $K$-spectral set of $A$ with $K\leq 11.08$; this was later improved  to $K=1+\sqrt{2}$, \cite{CP2017}. An elegant proof of  Crouzeix \& Palencia $(1+\sqrt{2})$-bound, \cite{RS2018} is included in an appendix. It follows,  in particular, that for all polynomials $p$,  
\be\label{eq:Cr}
\|p(A)\|_H \leq (1+\sqrt{2})\max_{z\in W_H(A)}|p(z)|.
\ee 

\begin{theorem}[{\bf Stability of Runge-Kutta schemes}]\label{thm:stable-RK-schemes} Consider the $s$-stage explicit
RK method $\Ps(z)=\sum_{k=0}^s a_kz^k$,   associated with region of absolute stability \mbox{$\sA_s=\{z\, :\, |\Ps(z)|\leq 1\}$}. Then, the RK scheme
\[
\bu_{n+1}=\Ps(\DLN)\bu_n, \qquad n=0,1,2, \ldots.
\]
 is stable under the  CFL condition $\Delta t W_{\HN}\!(\LN) \subset \sA_s$,
 \be\label{eq:s-stage-stability}
 \Delta t\,W_{\HN}\!(\LN) \subset \sA_s \ \ \leadsto \ \ |\bu_n|_{{}_{\ell^2}}\leq (1+\sqrt{2})\const_{\bbH}|\bu_0|_{{}_{\ell^2}}, \qquad n=1,2,\ldots.
 \ee
\end{theorem}
\noindent
For proof we apply \eqref{eq:Cr} with $p=\Ps^n$:
\[
\|\Ps^n(\DLN)\|_{\HN} \leq (1+\sqrt{2})\max_{z\in \Delta t W_{\HN}\!(\LN)}|\Ps^n(z)| \leq 
(1+\sqrt{2})\max_{z\in \sA_s}|\Ps^n(z)| \leq 1+\sqrt{2},
\]
and hence $\|\Ps^n(\DLN)\| \leq (1+\sqrt{2})\const_{\bbH}$.
\begin{remark}[{\bf Implicit RK methods}] The  argument above makes a critical use of the striking fact that the spectral set bound, $K=1+\sqrt{2}$, is independent of  neither the increasing degree, $\textnormal{deg}(\Ps^n)=sn$, nor  of the increasingly large dimension, $\textnormal{dim}(\LN)=N$.  In fact,   since \eqref{eq:Cr} applies  to the larger algebra of \emph{rational functions} bounded on $W_H(A)$,  theorem \ref{thm:stable-RK-schemes} can be equally well formulated to general \underline{implicit} RK methods, \cite[II.7]{HNW1993}.
\end{remark} 

 We recall the spectral stability analysis  \eqref{eq:absolute-stability} which is quantified in terms  of the  the spectrum $\sigma(\LN)$
\[
\Delta t\,\sigma(\LN) \subset \sA_s, \qquad \sigma(A):=\{\lambda_k(A)\, : \, k=1,2,\ldots,N\}. 
\]
In the terminology of \eqref{eq:spectral-set}, the spectrum $\sigma(\LN)$ is not a spectral set for $\LN$. Theorem  \ref{thm:stable-RK-schemes} tells us that replacing  the spectrum with the larger set of $H$-weighted numerical range, $W_{\HN}\!(\LN) \supset \sigma(\LN)$, provides a  very general framework for the stability of any Runge-Kutta scheme, in conjunction with  any $\LN$. For example, the forward Euler \eqref{eq:FE} applies to the one-sided difference \eqref{eq:one-sided-revisited} which was covered in Corollary \ref{cor:FE}. 
Observe that for \emph{normal} matrices\footnote{$\LN^*\LN=\LN\LN^*$ where $\LN^*$ is the $\ell^2$-adjoint of $\LN$.}, $\LN$, there holds $\textnormal{conv}\{\sigma(\LN)\}=W(\LN)$, e.g., \cite{Hen1962}. Thus, the gap $W_{\HN}\!(\LN)\backslash\textnormal{conv}\{\sigma(\LN)\}$ comes into play in the stability statement \eqref{eq:s-stage-stability} when normality uniform-in-$N$ fails --- precisely the scenario described in  section \ref{sec:spectral-not-enough} for failure of spectral analysis to secure stability.\newline
A main drawback of the CFL condition \eqref{eq:s-stage-stability} is its formulation in terms of a weighted numerical range  which is not always easily accessible. Here comes the imaginary interval condition, \eqref{eq:along-iaxis}, which provides an accessible sufficient  condition for stability of multi-stage RK methods.
\begin{theorem}[{\bf Stability of Runge-Kutta methods}]\label{thm:stable-RK-methods}
Consider the $s$-stage explicit
RK method   and assume it satisfies the imaginary interval condition \eqref{eq:along-iaxis}, namely --- there exists $R_s>0$ such that
\be\label{eq:IIC-revisited}
  \max_{-R_s\leq \sigma \leq R_s}|\Ps(i\sigma)|\leq 1,  \qquad \Ps(z)=1+z+ a_2z^2 + \ldots + a_sz^s.
\ee
Then, there exists  a consonant $0<\CFL_s<R_s$ such that  for all negative $\LN$\!\!'s, \eqref{eq:negative}, the RK method 
\[
\bu_{n+1}=\Ps(\DLN)\bu_n, \qquad n=0,1,2, \ldots, 
\]
  is stable under the  CFL condition $\Delta t \cdot r_{{}_{\HN}}\!(\LN) \leq \CFL_s$, 
 \be\label{eq:s-stage-stability-method}
 \Delta t\cdot r_{{}_{\HN}}\!(\LN) \leq \CFL_s \ \ \leadsto \ \ |\bu_n|_{{}_{\ell^2}}\leq (1+\sqrt{2})\const_{\bbH}|\bu_0|_{{}_{\ell^2}}, \qquad n=1,2,\ldots.
 \ee
\end{theorem}

\begin{proof} Recall $B_{\alpha}^-$ denotes the semi-disc,
$B_{\alpha}^-:=\{z :\, \textnormal{\Real}z\leq 0, \ \ |z|\leq \alpha\}$.
Consider an arbitrary negative $\LN$,
\[
2\textnormal{\Real} \langle \LN\bx,\bx\rangle_{\HN} = \langle \LN^\top\HN+\HN\LN\bx,\bx\rangle \leq 0
\]
The negativity of $\LN$ states that the weighted numerical range $W_{\HN}\!(\LN)$ lies on the left side of complex plane, 
and in fact, inside the left semi-disc
\[
W_{\HN}\!(\LN) \subset B_{r_{{}_{\HN}}\!(\LN)}^-:=\{z\, :\, \textnormal{\Real}z\leq 0, \ \ |z|\leq r_{{}_{\HN}}\!(\LN)\}.
\]
Next, we make use of \cite[Theorem 3.2]{KS1992} which asserts\footnote{Note that this requires $\Ps(0)=\Ps'(0)=1$ in \eqref{eq:IIC-revisited}.} that 
for an $s$-stage RK method satisfying the  imaginary interval condition, its region of absolute stability  contains a non-trivial semi-disc $B_{\CFL_s}^{-}$ with $\CFL_s\leq R_s$,  so that
\be\label{eq:BCsinAs}
\sA_s \supset B_{\CFL_s}^-:=\{z\, :\, \textnormal{\Real}z\leq 0, \ \ |z|\leq \CFL_s\}, \qquad \CFL_s\leq R_s.
\ee
We conclude that for small step-size \eqref{eq:s-stage-stability-method}
\[
\Delta t\,W_{\HN}\!(\LN) \subset \Delta t\, B_{r_{{}_{\HN}}\!\!(\LN)}^- = B_{ \Delta t \cdot r_{{}_{\HN}}\!\!(\LN)}^- \subset  B_{\CFL_s}^- \subset \sA_s.
\]
 Theorem  \ref{thm:stable-RK-schemes} implies stability \eqref{eq:stable-RK-scheme} with  $\const_{\bbL}=(1+\sqrt{2})\const_{\bbH}$.
\end{proof}
\begin{remark} We note that theorem \ref{thm:stable-RK-methods} makes use of the semi-disc $B^-_{\CFL_s}$ as a spectral set for $\Ps(\DLN)$. In this case, one expects a sharper  bound, compared with \eqref{eq:Cr}, 
\cite[\S3.2]{SdV2023},   $\|p(A)\|_H \leq 2\max_{z\in W_H(A)}|p(z)|$. The constant 2  agrees with Crouzeix's conjecture \cite{Cro2007} regarding the  optimality of the numerical range as $2$-spectral set.  
\end{remark}
\subsection{Optimality of the numerical radius-based CFL condition} We observe that the CFL condition quoted in \eqref{eq:s-stage-stability-method}, 
\be\label{eq:CFL-optimal}
\Delta t\cdot r_{{}_{\HN}}\!\!(\LN) \leq \CFL_s,
\ee
 offers  a refinement of the CFL condition \eqref{eq:CFL}. Indeed, 
since $\HN$ is uniformly bounded $0<\const^{-1}_{\bbH}\leq \HN\leq \const_{\bbH}$, we have 
\[
r_{{}_{\HN}}\!\!(\LN) \leq \|\LN\|_{\HN} \leq \const_{\bbH}\|\LN\|,
\]
and hence, the CFL condition --- compare with \eqref{eq:CFL},
$\Delta t\cdot\|\LN\|\leq \CFL'_s$ with $\CFL'_s:=\CFL_s/\const_{\bbH}$,
  implies that \eqref{eq:s-stage-stability-method} holds, and stability follows.\newline
 In fact, we claim that  \eqref{eq:CFL-optimal} offers an optimal CFL condition in the following sense. The proof of theorem \ref{thm:stable-RK-methods} compares   two semi-discs: on one hand we identified $B_{\calC_s}^-$ as the largest semi-disc inscribed inside $\sA_s$ (this is a property of the RK method under consideration); on the other hand, we identified  $B^-_{r_{{}_{\HN}}\!\!(\LN)}$ as the smallest semi-disc which contains $W_{\HN}\!(\LN)$.
 The CFL condition \eqref{eq:CFL-optimal} secures the dilation of the latter semi-disc inside the former, and there, we  seek  the smallest semi-disc associated with $\LN$ which satisfies a set of desired requirements.
 We claim that we cannot find a smaller semi-disc which will secure this line of argument. Indeed, let $\vnorm{\cdot}$ denote an arbitrary  (vector) norm on $M_N(\bbC)$,  with a semi-disc $B^-_{\vnorm{\LN}}$ which would be a candidate  for a better CFL condition, i.e., an even smaller semi-disc $B^-_{\vnorm{\LN}}\subset B^-_{r_{{}_{\HN}}\!\!(\LN)}$.  Clearly, by the necessity encoded in \eqref{eq:absolute-stability}, the CFL condition requires that $\vnorm{A}$ is \emph{spectrally dominant} in the sense that $\vnorm{A}\geq |\lambda_{\textnormal{max}}(A)|$ for all $A\in M_N(\bbC)$.
 Moreover, since power-boundedness is invarainat under unitary transformations,
 $\|(UAU^*)^n\|_{\HN}=\|A^n\|_{\HN}$, we ask that the semi-disc  associated with $\vnorm{\cdot}$ be unitarily invariant, 
 \[
 UB^-_{\vnorm{\LN}}U^*=B^-_{\vnorm{\LN}} \ \  \textnormal{for all} \ \ U'\textnormal{s} \ \ \textnormal{such that} \ \ |U\bx|_{\HN}=|\bx|_{\HN}.
 \]
  It follows from the main theorem of \cite{FT1984} that the semi-disc $B^-_{\vnorm{\LN}}$ must contains $B^-_{r_{{}_{\HN}}\!\!(\LN)}$. That is, the corresponding CFL condition \eqref{eq:CFL-optimal} is optimal in the sense that it is the smallest, spectrally dominant, unitarily invariant semi-disc which 
 makes the argument of theorem  \ref{thm:stable-RK-methods} work.
  
A main aspect of theorem \ref{thm:stable-RK-methods} is  going beyond any specific coercivity requirement which was sought in the SSP-based arguments in section \ref{sec:numerical-RKs}. It applies to \emph{all} negative $\LN$\!\!'s, thus addressing the question sought in \cite[\S3.5]{LT1998}.
A precise characterization for  RK methods satisfying the imaginary interval condition was given in \cite[Theorem 3.1]{KS1992}. Consider an explicit $s$ stage RK method, accurate of order $r\geq 1$,
\begin{subequations}\label{eqs:RK-IIC} 
\be\label{eq:RK-IIC}
\Ps(z)=\sum_{k=0}^r \frac{z^k}{k!} + \sum_{k=r+1}^s a_kz^k, \qquad r\geq 1.
\ee
It satisfies the imaginary interval condition \eqref{eq:along-iaxis} if and only if
\be\label{eq:interval-condition}
\left\{
\begin{array}{ll}
\displaystyle (-1)^{\frac{r+1}{2}}(a_{r+1}-1) <0, & r \ \textnormal{is odd},\\ \\
\displaystyle (-1)^{\frac{r+2}{2}}(a_{r+2}-(r+2)a_{r+1}+r+1) <0, & r \ \textnormal{is even}.
\end{array}
\right.
\ee
\end{subequations}
In the particular case of $s=r=3,4$ we find that the 3-stage RK method \eqref{eq:RK3} and 4-stage RK method \ref{eq:RK4} satisfy the  imaginary interval condition and hence  the existence of semi-discs with radii
$\CFL_3=\sqrt{3}$ and $\CFL_4=2.61$, shown in figure \ref{fig:semi-circle3+4} which imply stability under the  respective CFL conditions, 
\[
\Delta t \cdot\|\LN\|\leq \CFL'_s, \qquad \CFL'_s=\CFL_s/\const_{\bbH}.
\]
 In particular, this extends the strong stability statement of 3-stage \eqref{eq:RK3} in \cite[Theorem 2]{Tad2002} and provides the first stability proof for the 4-stage RK \eqref{eq:RK4} for arbitrarily large systems.\newline
 Condition \eqref{eq:interval-condition} becomes more restrictive  for higher order methods; instead, one can increase $r$ and use $s$-stage protocol, $s>r$ to form  a dissipative term  $\sum_{k=r+1}^s a_kz^k$ which enforces the  imaginary interval condition. In particular,  the 7-stage Dormand-Prince  method \cite{DP1980}, with embedded fourth- and fifth-order accurate RK45,  $(r,s)=(5,7)$ which is used in MATLAB, does satisfy the imaginary interval condition \eqref{eq:interval-condition}, \cite{SR1997}. See the example  of the 10-stage explicit RK method SSPRK(10,4) in \cite[Fig. 2]{RO2018}.
 
 \begin{figure}
 \begin{eqnarray*}
\includegraphics[scale = 0.62]{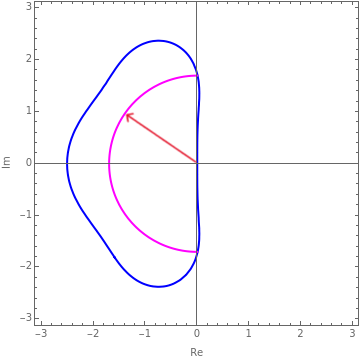} &
\includegraphics[scale = 0.62]{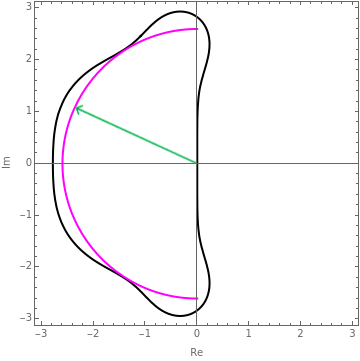}
\end{eqnarray*}
\caption{The semi-circles $B^-_{\CFL_s}(0)$ inscribed inside $\sA_3$ (left) and $\sA_4$ (right),}\label{fig:semi-circle3+4}
\end{figure}

\section{Stability of  time-dependent methods of lines}\label{sec:results}
We demonstrate application of the new stability results  for arbitrarily large systems in the context of methods of lines for difference approximation of the scalar hyperbolic equation
\[
y_t=a(x)y_x, \quad (t,x)\in {\mathbb R}_+\times [0,1],
\] 
augmented with proper boundary conditions. The stability results extend, mutatis mutandis\footnote{In particular, $\ell^2$-stability needs to be adjusted to weighted $H$-stability,  weighted by the smooth symmetrizer $H=H(x,\xi)$ so that $H(x,\xi)\sum_jA_j(x)e^{ij\xi}$ is symmetric.}, to multi-dimensional hyperbolic problems,  $\by_t =\sum_{j=1}^d A_j(x)\by_{x_j}$. Stability theories for such difference approximations were developed in the classical works in the 50s--70s, e.g., \cite{LR1956,LW1964,LN1964,Kre1964,RM1967,Kre1968,GKS1972} and can be found in the more recent texts of \cite{LeV2007,GKO2013,Hes2017}. 
Our aim here is to revisit the question of stability for RK time-discretizations of such difference approximations, from a perspective of the stability theory developed in section \ref{sec:spectral-set}.   A central part of this approach   requires computation of the (weighted) numerical range of the large matrices that arise in the context of such difference approximations. 
The development of full stability theory along these lines is beyond the scope of this paper, and is left for future work.

\subsection{Periodic problems. Constant coefficients}\label{sec:periodic} We consider the $1$-periodic problem 
\[
\left\{\begin{split}
y_t(x,t)&=ay_x(x,t), \quad (t,x)\in {\mathbb R}_+\times [0,1]\\
y(0,t)&=y(1,t).
\end{split}\right.
\]
It spatial part is discretized using  finite-difference method with constant coefficients (depending on $a$),  $\{\diffq_\alpha\}$,   and acting on a discrete grid, $x_\nu=\nu\Delta x, \, \Delta x=\nicefrac{1}{N}$,
\[
\ddt y(x_\nu,t)=\diffQ(E)y(x_\nu,t), \quad \nu=0,1,\ldots, N-1, \qquad \diffQ(E):=\frac{1}{\Delta x}\sum_{\alpha=-\ell}^r \diffq_\alpha E^\alpha.
\]
Here $E$ is the $1$-periodic translation operator, $Ey_\nu=y_{(\nu+1)[mod N]}$. The resulting scheme amounts to a system of ODEs for the $N$-vector of unknowns, $\by(t)=\big(y(x_0,t),\ldots, y(x_{N-1},t)\big)^\top$, which admits the \emph{circulant} matrix representation
\be\label{eq:cE}
\dot{\by}(t)=\QEN \by, \quad \QEN=\frac{1}{\Delta x}\sum_{\alpha=-\ell}^r \diffq_\alpha\bbE^\alpha, \qquad \EN:=\left[\begin{array}{ccccc}
0& 1 & \ldots & \ldots & 0\\
0 & 0 & 1 & \ddots & \vdots\\
\vdots & \ddots & \ddots &  \ddots & \vdots\\
0 & \ldots & \ddots &  0 & 1\\
1 & \ldots & \ldots & 0 & 0\end{array}\right]_{N\times N}.
\ee
The numerical range of circulant matrices is given by  convex polytopes.
Indeed, let $\bbF$ denote the unitary Fourier matrix, $\displaystyle \bbF_{jk}=\Big\{\frac{1}{\sqrt{N}}e^{\nicefrac{2\pi ijk}{N}}\Big\}_{j,k=1}^N$.
Then $\bbF$ diagonalizes $\EN$,
\[
\langle \EN\bx,\bx\rangle=\langle\widehat{\bbE}_N\widehat{\bx},\widehat{\bx}\rangle, \quad \widehat{\bbE}_N:=\bbF^*\EN\bbF=\left[\begin{array}{ccccc}
e^{\frac{2\pi i}{N}}& 0 & \ldots & \ldots & 0 \\
0 & e^{2\frac{2\pi i}{N}} & 0 & \ddots & \vdots\\
\vdots & \ddots & \ddots &  \ddots & \vdots\\
0 & \ldots & \ddots &  e^{(N-1)\frac{2\pi i}{N}} & 0\\
0 & \ldots & \ldots & \ldots & 1\end{array}\right], \quad \widehat{\bx}:=\bbF^*\bx.
\]
and hence $\displaystyle W(\EN)=\Big\{\sum_{j=1}^N |\widehat{x}_j|^2 e^{\nicefrac{2\pi ij}{N}} \, : \ 
\sum_j |\widehat{x}_j|^2=1\Big\}$ is the regular $N$-polytope with vertices 
at $\{e^{\nicefrac{2\pi ij}{N}}\}_{j=1}^N$. This should be compared with the numerical range of the Jordan block \eqref{eq:W-of-J}.\newline
It follows that $\widehat{\QEN}=\diffQ(\widehat{\bbE}_N)$ and hence the action of the $N\times N$ circulant $\QEN$ is encoded it terms of its \emph{symbol}, 
$\displaystyle \widehat{\diffq}(\xi):=\frac{1}{\Delta x}\sum_\alpha \diffq_\alpha e^{i\alpha\xi}$,
\[
\langle \QEN\bx,\bx\rangle=\langle\widehat{\QEN}\widehat{\bx},\widehat{\bx}\rangle, \qquad \widehat{\QEN}=\left[\begin{array}{ccccc}
\widehat{\diffq}(\frac{2\pi}{N})& 0 & \ldots & \ldots & 0 \\
0 & \widehat{\diffq}(2\frac{2\pi}{N}) & 0 & \ddots & \vdots\\
\vdots & \ddots & \ddots &  \ddots & \vdots\\
\vdots & \ddots & \ddots &  \widehat{\diffq}((N\!\!-\!\!1)\frac{2\pi}{N}) & 0\\
0 & \ldots & \ldots & 0 & \widehat{\diffq}(2\pi)\end{array}\right].
\]
\begin{lemma}[{\bf Numerical range of circulant matrices}]\label{lem:Wcirculant}
The numerical range of the circulant matrix  $\QEN$ is given by the convex polytope with vertices 
at $\{\widehat{\diffq}(\nicefrac{2\pi j}{N})\}_{j=1}^N$,
\[
W(\QEN)=\Big\{\sum_j |\widehat{x}_j|^2 \,\widehat{\diffq}\big(\nicefrac{2\pi j}{N}\big) \, : \ |\widehat{\bx}|=1\Big\}.
\]
\end{lemma}
We now appeal to theorem \ref{thm:coercive-FE} which secures the stability of forward Euler time discretization for $\LN=\QEN$, provided the CFL condition $\Delta t  W(\QEN) \subset B_1(-1)$ holds.
  \begin{proposition}[{\bf Stability --- difference  schemes with constant coefficients. I}]\label{prop:coercive-QE}
 Consider the fully-discrete finite difference scheme
 \[
 \bu_{n+1}=\bu_n + \frac{\Delta t}{\Delta x}\sum_\alpha \diffq_\alpha\EN^\alpha \hspace*{0.04cm}\bu_n, \quad n=0,1,2,\ldots.
 \]
 The scheme is stable under the CFL condition, 
 \be\label{eq:CFLQ}
 \max_{1\leq j\leq N}\Big|1+\Delta t \hspace*{-0.04cm}\cdot\hspace*{-0.04cm}\widehat{\diffq}(\nicefrac{2\pi j}{N})\Big|\leq 1, \qquad  \widehat{\diffq}(\xi):=
\frac{1}{\Delta x}\sum_\alpha \diffq_\alpha e^{i\alpha\xi},
 \ee 
 and the following stability bound  holds 
 $ |\bu_n|_{{}_{\ell^2}}\leq 2|\bu_0|_{{}_{\ell^2}}, \ \forall n\geq 1$.
  \end{proposition}
Since the CFL condition \eqref{eq:CFLQ} guarantees that $\LN=\bbI+\DQEN$ is coercive, the result goes over to SSP-based multi-stage RK time differencing. In fact, theorem \ref{thm:stable-RK-methods} applies for multi-stage RK time differencing and for \underline{all} negative $\QEN$'s.
\begin{proposition}[{\bf Stability --- difference  schemes with constant coefficients. II}]\label{prop:negative-QE} 
Consider the  fully-discrete finite difference scheme
\be\label{eq:negative-QE}
\begin{split}
\bu_{n+1}=\Ps\big(\DQEN\big)\bu_n, \ \  & n=0,1,2, \ldots, \  \\
 & \Ps(z)\!=\!\sum_{k=0}^s a_kz^k, \quad    Q(\EN)\!=\!\frac{1}{\Delta x}\sum_\alpha \diffq_\alpha \EN^\alpha.
 \end{split}
\ee
Here, $\Ps$ is an $s$-stage RK stencil  satisfying the imaginary interval condition, so that \eqref{eq:BCsinAs} holds with $\CFL_s>0$.
If the spatial discretization is neagtive, $\Real \widehat{\diffq}(\nicefrac{2\pi j}{N})\leq 0$, then the scheme \eqref{eq:negative-QE} is stable under the CFL condition
 \be\label{eq:CFL-negative-QE}
\max_{1\leq j\leq N}|\Delta t\cdot \widehat{\diffq}(\nicefrac{2\pi j}{N})|\leq \CFL_s,\qquad \widehat{\diffq}(\xi)=Q(e^{i\xi}),
\ee
 and the following stability bound  holds,
 \[
 |\bu_n|_{{}_{\ell^2}}\leq (1+\sqrt{2})|\bu_0|_{{}_{\ell^2}}, \qquad n=1,2,\ldots.
 \]
\end{proposition} 
Propositions \ref{prop:coercive-QE} and \ref{prop:negative-QE} recover  von-Neumann stability analysis for difference schemes with constant coefficients, \cite[\S4.2]{GKO2013}. We shall consider three examples.\newline
\begin{example}[{\bf One-sided differences}] Consider the periodic setup of the one-sided difference \eqref{eq:transport}, 
\[
\bu_{n+1}=\big(\bbI+ \DQEN)\big)\bu_n, \qquad Q(\EN)=\frac{a}{\Delta x} (\EN-{\mathbb I}),
\]
 with spatial  symbol 
$\displaystyle \widehat{\diffq}(\xi)=\frac{a}{\Delta x}(e^{i\xi}-1)$.
This  amounts to the $N\times N$ system 
\[
\bu_{n+1}=\LN\bu_n, \qquad \LN=\left[\begin{array}{ccccc}
1\!-\!\mesh a& \mesh a & 0 & \ldots & 0\\
0 & 1\!-\!\mesh a & \mesh a & \ddots & \vdots\\
\vdots & \ddots & \ddots &  \ddots & 0\\
0 & \ddots & \ddots &  1\!-\!\mesh a & \mesh a\\
\mesh a & 0 & \ldots & 0 & 1\!-\!\mesh a\end{array}\right]_{N\times N}, \quad \mesh=\frac{\Delta t}{\Delta x}.
\]
 Using proposition \ref{prop:coercive-QE} we secure stability under the usual CFL condition $\mesh a\leq 1$,
\[
\mesh a\leq 1 \ \leadsto \ \max_{1\leq j\leq N}\big|1+\mesh a(e^{\nicefrac{2\pi ij}{N}}-1)\big|^2=|1-\mesh a|^2+2|1-\mesh a|\mesh a +(\mesh a)^2\leq 1.
\]
This extends to multi-stage time differencing, RKs, \ $s,3,4$
\[
\bu_{n+1}=\Ps(\DQEN)\bu_n, \qquad \Ps(z)=\sum_{k=0}^s \frac{z^k}{k!}, \ s=3,4.
\]
Clearly, $\Real \widehat{\diffq}\leq0$, and we can appeal to proposition \ref{prop:negative-QE} which secures stability under CFL condition $\mesh a\leq \CFL_s$; indeed,
\[
\mesh a\leq \CFL_s \ \ \leadsto \ \ \mesh a (e^{\nicefrac{2\pi i j}{N}}-1) \in B^-_{\CFL_s}, \quad j=1,2,\ldots, N.
\]
\end{example}
\begin{example}[{\bf Centered differences}] Consider the periodic setup of the centered spatial difference scheme, \eqref{eq:centered}, combined with multi-stage RK time differencing, RKs, \ $s=3,4$,
\[
\bu_{n+1}=\Ps(\DQEN)\bu_n, \qquad Q(\EN)=\frac{a}{2\Delta x} (\EN-\EN^{-1}).
\] 
Spatial differencing has purely imaginary symbol $\displaystyle \widehat{\diffq}(\xi)=\frac{a}{\Delta x}i\sin(\xi)$, and  we invoke proposition \ref{prop:negative-QE} which  secures stability under the CFL condition \eqref{eq:CFL-negative-QE},
\[
\mesh a =  \max_{1\leq j\leq N}\Big|\mesh a i\sin(\nicefrac{2\pi ij}{N}\Big|\leq \CFL_s, \qquad \mesh=\frac{\Delta t}{\Delta x}.
\]
This line of argument extends to higher order centered differences,  \cite[\S5.2]{Tad2002}, e.g., the fourth-order difference
\be\label{eq:5-point-stencil}
Q(\EN)=\frac{a}{12\Delta x} (-\EN^2+8\EN-8\EN^{-1}+\EN^{-2})
\ee
or the fourth-order finite-element difference
\[
Q(\EN)=\left[\begin{array}{ccccc}
\nicefrac{4}{6}& \nicefrac{1}{6} & 0 & \ldots & \nicefrac{1}{6}\\
\nicefrac{1}{6} & \nicefrac{4}{6} & \nicefrac{1}{6} & \ddots & 0\\
0 & \ddots & \ddots &  \ddots & \vdots\\
\vdots & \ddots  & \ddots &  \nicefrac{4}{6} & \nicefrac{1}{6}\\
\nicefrac{1}{6} & 0 & \ldots & \nicefrac{1}{6} & \nicefrac{4}{6}\end{array}\right]^{-1}\hspace*{-0.2cm}\times \ \frac{1}{2\Delta x}\left[\begin{array}{ccccc}
0& 1 & \ldots & \ldots & -1\\
-1 & 0 & 1 & \ddots & \vdots\\
0 & \ddots & \ddots &  \ddots & \vdots\\
\vdots & \ddots & \ddots &  0 & 1\\
1 & 0 & \ldots & -1 & 0\end{array}\right]_{N\times N}.
\] 
\end{example}
\begin{example}[{\bf Lax-Wendroff differencing}] We use the Lax-Wendroff protocol for second-order spatial difference \cite{LW1964} (observe that  the mesh ratio, $\mesh=\nicefrac{\Delta t}{\Delta x}$, is kept fixed),
\[
Q_{{}_{\textnormal{LW}}}(\EN)= \frac{a}{2\Delta x} (\EN-\EN^{-1}) +\frac{\mesh a^2}{2\Delta x}(\EN-2\bbI+\EN^{-1}),
\]
with symbol 
\[
\widehat{\diffq}_{{}_{\textnormal{LW}}}(\xi)=\frac{a}{\Delta x}i\sin(\xi) + \frac{\mesh a^2}{\Delta x}(\cos(\xi)-1).
\]
Stability of the Lax-Wendroff (LW) scheme 
\be\label{eq:LW-scheme}
\bu_{n+1}=\big(\bbI+\Delta t Q_{{}_{\textnormal{LW}}}(\EN)\big)\bu_n
\ee
follows provided  CFL condition \eqref{eq:CFLQ} holds, namely $|1+\Delta t\widehat{\diffq}_{{}_{\textnormal{LW}}}(\nicefrac{2\pi j}{N})|\leq 1$. Noting that 
\[
\widehat{\diffq}_{{}_{\textnormal{LW}}}(\xi) = \frac{2a}{\Delta x}i\sin(\nicefrac{\xi}{2})\cos(\nicefrac{\xi}{2}) - \frac{2\mesh a^2}{\Delta x}\sin^2(\nicefrac{\xi}{2}),
\]
it is a standard argument, e.g., \cite[\S1.2]{GKO2013} to conclude that $\mesh a\leq 1$ secures the desired CFL condition,
\[
 \mesh a\leq 1 \ \leadsto \   \max_{1\leq j\leq N}|1+\Delta t \widehat{\diffq}_{{}_{\textnormal{LW}}}(\nicefrac{2\pi j}{N})|^2\leq 1.
\]
We note that LW differencing has a negative symbol 
$\displaystyle \Real \widehat{\diffq}_{{}_{\textnormal{LW}}}(\xi) \leq 0$,
and therefore theorem \ref{thm:stable-RK-methods} secures the stability of higher-order time discretizations of LW scheme
\[
\bu_{n+1}=\Ps\big(\Delta t Q_{{}_{\textnormal{LW}}}(\EN)\big)\bu_n, \quad s=3,4,
\]
under the relaxed CFL condition, $2\mesh a\leq \CFL_s$. Indeed, 
\[
2\mesh a\leq \CFL_s \ \leadsto \ \max_{1\leq j\leq N}\big|\Delta t\, \widehat{\diffq}_{{}_{\textnormal{LW}}}(\nicefrac{2\pi j}{N})\big|^2 \leq \max_{\xi}\big\{4\big(\mesh a \sin(\nicefrac{\xi}{2})\cos(\nicefrac{\xi}{2})\big)^2 +4 \big(\mesh a \sin(\nicefrac{\xi}{2})\big)^4\big\}\leq \CFL^2_s.
\]
\end{example}

\smallskip
The constant coefficient case in the period setup involves the algebra of circulant matrices,  all of which are uniformly diagonlizable by the Fourier matrix $\bbF$. This is a rather special case, in which  von Neumann spectral stability analysis prevails for arbitrarily large systems. Clearly, the  numerical range-based stability results of sections \ref{sec:range} and \ref{sec:spectral-set} offer a more general framework for studying stability of general non-periodic  cases. Examples are outlined below. 
 
\subsection{Periodic problems. Variable coefficients}\label{sec:periodic-var} We consider the $1i$-periodic problem with $C^2$-variable coefficient $a(\cdot)$
\be\label{eq:yt=axyx}
\left\{\begin{split}
y_t(x,t)&=a(x)y_x(x,t), \quad (t,x)\in {\mathbb R}_+\times [0,1]\\
y(0,t)&=y(1,t).
\end{split}\right.
\ee
The spatial part is discretized using  finite-difference method with $a(x)$-dependent variable coefficients,  $\{\diffq_\alpha(x)\}$,   and acting on a discrete grid, $x_\nu=\nu\Delta x, \, \Delta x=\nicefrac{1}{N}$,
\be\label{eq:variable-coeff}
\ddt y(x_\nu,t)=\diffQ(E)y(x_\nu,t), \quad \nu=0,1,\ldots, N-1, \qquad \diffQ(E):=\frac{1}{\Delta x}\sum_{\alpha=-\ell}^r \diffq_\alpha(x) E^\alpha.
\ee
The accuracy requirement places the restriction $\sum_\alpha \diffq_\alpha(x)=0, \ \sum_\alpha \alpha \diffq_\alpha(x)=a(x)$ and so on.
The difference scheme \eqref{eq:variable-coeff} amounts to an $N\times N$  system of ODEs with `slowly varying' circulancy, that is $Q(x,\EN)_{ij}$ changes smoothly in the sense that  
$|Q(x,\EN)_{i+1,j+1}-Q(x,\EN)_{ij}|$ is bounded independent of $1/\Delta x$.  
\be\label{eq:local-stencil}
\Delta x \sum_\alpha \alpha^2 |\diffq_\alpha(x)|_{C^2}\leq K_{\diffq}.
\ee

Let $\widehat{Q}$ denote the formal symbol associated with \eqref{eq:variable-coeff}
\[
\widehat{Q}(x,\xi):= \frac{1}{\Delta x}\sum_{\alpha=-\ell}^r \diffq_\alpha(x) e^{i\alpha\xi}.
\]
Assume that the symbol is negative 
$\Real \widehat{Q}(x,\xi)\leq 0$.
Then by the sharp G\r{a}rding inequality, \cite[Theorem 1.1]{LN1964}, see also \cite{LW1962}, 
the corresponding difference operator is semi-bounded\footnote{Note that $Q(x,\EN)$ is unbounded, $\|Q(x,\EN)\| = {\mathcal O}(\nicefrac{1}{\Delta x})$.}, namely ---  there exists a constant $\eta>0$ depending of $K_{\diffq}$ but otherwise independent of $N$, such that 
\be\label{eq:Garding}
\Real Q(x,\EN) \leq 2\eta \bbI_{N\times N}.
\ee
Theorem \ref{thm:stable-RK-methods} applies to $Q(x,\EN)-\eta\bbI$, implying its power-boundedness under the CFL condition \eqref{eq:CFL},
\[
\|\Ps^n\big(\Delta t(Q(x,\EN)-\eta\bbI)\big)\|\leq 1+\sqrt{2}, \qquad \Delta t\cdot r\big(Q(x,\EN)\big)\leq \CFL_s.
\]
Next, we note that the shift $-\eta\bbI$ produces only a finite bounded perturbation $B$, namely
\[
\begin{split}
\Ps\big(\Delta t\cdot Q(x,\EN)\big)&= \Ps\big(\Delta t\cdot (Q(x,\EN)-\eta\bbI)+\Delta t\cdot\eta \bbI\big)\\
 & = \Ps\big(\Delta t\cdot (Q(x,\EN)-\eta\bbI)\big) + B, \quad B=\Delta t\cdot \eta \sum_{k=1}^s a_k k \big(\Delta t \cdot Q(x,\EN)\big)^{k-1},
\end{split}
\]
with 
$\displaystyle \|B\|\leq 2\Delta t\cdot \eta K, \  K:=\sum_{k=1}^s |a_k|k \CFL_s^{k-1}$, which in turn implies the stability bound 
\[
|\bu(t_n)|\leq \|\Ps^n\big(\Delta t\cdot Q(x,\EN)\big)\||\bu_0|\leq 
(1+\sqrt{2})e^{\eta K t_n}.
\]
We summarize by stating
\begin{proposition}[{\bf Stability --- finite difference schemes with variable coefficients}]\label{prop:variable}
Consider the  fully-discrete finite difference scheme
\be\label{eq:negative-QE-variable}
\bu_{n+1}=\Ps\big(Q(x,\EN)\big)\bu_n, \ \  n=0,1,2, \ldots, 
\ee
where $\displaystyle  Q(x,\EN)=\frac{1}{\Delta x}\sum_\alpha \diffq_\alpha(x) \EN^\alpha$ is a local  difference operator, \eqref{eq:local-stencil},
and $\Ps$ is an $s$-stage  RK stencil  satisfying the imaginary interval condition,  \eqref{eq:BCsinAs}.
If the spatial symbol is negative, 
\be\label{eq:negative-symbol}
\Real \widehat{Q}(x,\xi)\leq 0, \qquad \widehat{Q}(x,\xi):=\frac{1}{\Delta x}\sum_\alpha \diffq_\alpha(x) e^{i\alpha\xi},
\ee
then the scheme \eqref{eq:negative-QE-variable} is stable under the CFL condition
 \be\label{eq:CFL-negative-QxE}
\max_{\xi}|\Delta t\cdot \widehat{Q}(x,\xi)|\leq \CFL_s,
\ee
 and the following stability bound  holds with $K=\sum_{k=1}^s |a_k|k \CFL_s^{k-1}$,
 \[
 |\bu_n|_{{}_{\ell^2}}\leq (1+\sqrt{2})e^{\eta K t_n}|\bu_0|_{{}_{\ell^2}}, \quad n=1,2,\ldots, \quad \Delta t\cdot r\big(Q(x,\EN)\big)\leq \CFL_s.
 \]
\end{proposition} 

\noindent
{\bf Stability of Fourier method}. There are two approaches to handle the stability of difference approximations of problems with variable coefficients: 
the von-Neumann spectral analysis based on sharp G\r{a}rding inequality \eqref{eq:Garding}, or the energy method e.g., \cite[\S2]{Tad1987}; both approaches requires \emph{local} stencils \eqref{eq:local-stencil}.  An alternative approach for stability with variable coefficients  in based on \emph{numerical dissipation}, \cite{Kre1964}.
As an extreme example for using our  RK stability result, we consider the \emph{Fourier method}, \cite[\S4]{KO1972},\cite{GO1977}, which is neither local nor dissipative. Set $\Delta x=\nicefrac{1}{(2N+1)}$ with an odd number of $(2N+1)$ gridpoints. The Fourier method for \eqref{eq:yt=axyx}  amounts to  $(2N\!+\!1)\times (2N\!+\!1)$ system of ODEs
\be\label{eq:Fourier-method}
 \dot{\by}(t)=Q(\DFN)\by(t), \qquad Q(\DFN)=A\DFN, \quad  A=\left[\begin{array}{cccc}a(x_0) & & \\ & a(x_1) & \\ & & \ddots & \\ & & & a(x_{2N})\end{array}\right],
\ee  
where the  diagonal matrix  $A$  encodes $a(x)$ and $\DFN$ is the $(2N\!+\!1)\times (2N\!+\!1)$ Fourier differencing matrix 
\[
\DFN= \bbF\left[\begin{array}{ccccc}-iN & 0 & \ldots & & 0\\
 0 & \hspace*{-0.2cm} -i(N\!-\!1) & 0 & \ddots & \\
 \vdots & & \ddots & \ddots & \vdots\\
 \vdots & & \ddots & \hspace*{-0.2cm} i(N\!-\!1) & 0\\
  0 & \ldots & \ldots & 0 & \hspace*{-0.2cm} iN\\
  \end{array}\right]\bbF^*, \qquad \bbF_{jk}=\Big\{\frac{e^{ijk \Delta x }}{\sqrt{2N+1}}\Big\}_{j,k=1}^{2N+1}.
\]
The Fourier difference method is neither local,  $\ds (\DFN)_{jk}= \frac{(-1)^{j-k}}{2\sin(\nicefrac{(k-j)\Delta x}{2})}$ fails \eqref{eq:local-stencil}, nor dissipative, and  the  method is \underline{unstable} in presence of  variable coefficients, \cite{GHT1994}. However, there is a different weighted-stability. Specifically --- for the prototypical case $a(x)=\sin (x)$, there exists  a symmetrizer $\HN$ such that  \cite[Theorem 2.1]{GHT1994}
\[
 Q(\DFN)^\top\HN + \HN Q(\DFN) \leq \HN,
 \]
where the $\HN$-norm  corresponds to the $H^1$-norm 
 \[
 |\bu|^2_{\HN}=|\bu|^2_{H^1}, \qquad |\bu|^2_{H^s}:=\sum_{k=-N}^N |(1+k^2)^{\frac{s}{2}}|\widehat{u}_k|^2.
 \]
 \begin{proposition}[{\bf Stability --- Fourier method}]\label{prop:Fourier}
Consider the time discretization of the Fourier method, 
\[
 \dot{\by}(t)=Q(\DFN)\by(t), \qquad Q(\DFN)=A\DFN, \quad  A=\left[\begin{array}{cccc}\sin(x_0) & & \\ & \sin(x_1) & \\ & & \ddots & \\ & & & \sin(x_{2N})\end{array}\right],
\]
using RK methods which satisfy the imaginary interval condition, 
 \[
 \bu_{n+1}=\Ps\big(\Delta t\cdot Q(\DFN)\big)\bu_n, \quad n=1,2,\ldots, \qquad \Delta t\cdot  N\leq \CFL_s.
 \]
The Fourier method is $H^1$-stable 
\be\label{eq:H1-Fourier}
|\bu_n|_{\HN} \leq (1+\sqrt{2})e^{t_n/2}|\bu_0|_{\HN}.
\ee
\end{proposition}
We note that the symmetrizer  $\HN$  is not uniformly bounded from below,
$N^{-2}\bbI\leq \HN\leq 4\bbI$, so $\ell^2$-stability fails. 
 Converted to $\ell^2$-framework, \eqref{eq:H1-Fourier} yields
\[
|\bu_n|_{\ell^2} \leq N|\bu_n|_{\HN} \leq N(1+\sqrt{2})|e^{t_n/2}|\bu_0|_{\HN}=2N(1+\sqrt{2})|e^{t_n/2}|\bu_0|_{\ell^2}.
\]

\subsection{Initial-boundary value problems}\label{sec:IBVP}
We consider the  problem \eqref{eq:yt=ayx} in the the strip 
\[
\left\{\begin{split}
y_t(x,t)&=ay_x(x,t), \ \  a>0, \qquad (t,x)\in \bbR_+\times [0,1]\\
y(1,t)&=0.
 \end{split}
 \right.
 \]
A general stability theory for difference approximations of  initial-boundary value problems was developed in \cite{Kre1968,GKS1972}. It is based on normal mode analysis and secures the resolvent-type  stability of such approximations.
The following example shows how to utilize the framework offered in theorem  \ref{thm:stable-RK-schemes}, to study the stability of difference approximations of  initial-boundary value problems.  
\begin{example}[{\bf One-sided difference}]\label{exm:IBVP-a} Consider an interior  centered differencing augmented with one-sided difference at the outflow boundary $x=0$,
\be\label{eq:IBVP-1st-order}
\left\{\begin{split}
\ddt y(x_0,t)& = a\, \frac{y(x_1,t)-y(x_0,t)}{\Delta x} \\
\ddt y(x_\nu,t)&=  a\, \frac{y(x_{\nu+1},t)-y(x_{\nu-1},t)}{2\Delta x}, \qquad  \nu=1, 2, \ldots, N-1\\
 y(x_N,t)&=0.
\end{split}\right.
\ee
We emphasize that we treat the semi-infinite problem, which  
 amounts to method of lines  for the $N$-vector of unknowns, $\by(t):=\big(y(x_0,t), y(x_1,t), \ldots, y(x_{N-1},t)\big)^\top$, governed by the semi-discrete system $($``method of lines''$)$
\be\label{eq:IBVP-transport-a}
\dot{\by}(t)=\LN \by(t), \qquad \LN=\frac{a}{\Delta x}\left[\begin{array}{cccccc}
 -1& 1 & 0 & \ldots & \ldots &  \ldots \\
-\nicefrac{1}{2} & 0 & \nicefrac{1}{2} & \ddots & \ddots & \vdots\\
0 & \ddots & \ddots &  \ddots & &  \vdots\\
\vdots & \ddots & -\nicefrac{1}{2} &  \ddots & \ddots & 0 \\
 \vdots & \ldots  & \ddots & -\nicefrac{1}{2} & 0 &\nicefrac{1}{2} \\
\ldots  & \ldots & \ldots & 0 & -\nicefrac{1}{2}& 0
 \end{array}\right].
\ee
Although the matrix $\LN$ is not negative, $\displaystyle \LN^\top+\LN = \frac{a}{\Delta x}\left[\begin{array}{cc}-2 & \nicefrac{1}{2}  \\ \nicefrac{1}{2} & 0  \end{array}\right]
\oplus \ \, {}^{\bigzero} \hspace*{0.23cm}  {}_{{}_{(N-2)\times (N-2)}}$, 
it is weighted negative with the  simple symmetrizer $\HN$:
\[
\LN^\top\HN+\HN\LN = \frac{a}{\Delta x}\left[\begin{array}{rr}-1 & 0  \\ 0 & 0  \end{array}\right]
\oplus \ \, {}^{\bigzero} \hspace*{0.23cm}  {}_{{}_{(N-2)\times (N-2)}}
 \leq 0, \qquad \HN:=\left[\begin{array}{cc}
 \nicefrac{1}{2}& 0 \\
0 & 1  \end{array}\right]\oplus \bbI_{{}_{(N-2)\times (N-2)}}.
\]
\newline 
\ifx
Alternatively, we may augment the interior centered difference scheme with second-order accurate one-sided difference at the outflow boundary, $x=0$,
\be\label{eq:IBVP-2nd-order}
\left\{\begin{split}
\ddt y(x_0,t)& = a\, \frac{-y(x_2,t)+4y(x_1,t)-3y(x_0,t)}{2\Delta x} \\
\ddt y(x_\nu,t)&=  a\, \frac{y(x_{\nu+1},t)-y(x_{\nu-1},t)}{2\Delta x}, \qquad  \nu=1, 2, \ldots, N-1, \\
 y(x_N,t)&=0.
\end{split}\right.
\ee
This  amounts to 
\be\label{eq:IBVP-transport-b}
\dot{\by}(t)=\LN \by(t), \quad \LN=\frac{a}{2\Delta x}\left[\begin{array}{cccccc}
 -3& 4 & -1 & 0 & \ldots & \ldots\\
-1 & 0 & 1 & 0 & & \vdots\\
0 & \ddots & \ddots &  \ddots & \ddots & \vdots\\
\vdots & \ldots & -1 &  \ddots & \ddots & 0 \\
\vdots  & \ldots & \ldots & \ddots & \ddots & 1\\
\ldots & \ldots & \ldots & 0 & -1 & 0
 \end{array}\right].
 \ee
Again, $\LN$ is not negative, $\displaystyle 
\LN^\top+\LN= 
\frac{a}{2\Delta x}\left[\begin{array}{rrr}-6 & 3 & -1 \\ 3 & 0  &0\\
-1 & 0 & 0 \end{array}\right]
\oplus \  {}^{\bigzero} \hspace*{0.23cm}{}_{{}_{(N-3)\times (N-3)}}$,
 but there is a simple symmerizer, $\displaystyle \bbH:=\left[\begin{array}{ccc}
 \nicefrac{1}{4}& 0 & 0 \\
0 & 1 & 0 \\ 0 & 0 & 1 \end{array}\right]\oplus \bbI_{(N-3)\times (N-3)}$,
 which makes it weighted negative,
\[
\LN^\top\HN+\HN\LN = 
\frac{a}{2\Delta x}\left[\begin{array}{rrr}-\nicefrac{3}{2} & 0 & -\nicefrac{1}{4} \\ 0 & 0  &0\\
-\nicefrac{1}{4} & 0 & 0 \end{array}\right]
\oplus \ \ {}^{\bigzero} \hspace*{0.23cm}  {}_{{}_{(N-3)\times (N-3)}} 
\leq 0.
\]
\fi
Using theorem \ref{thm:stable-RK-methods}, we conclude the stability of  time discretization of \eqref{eq:IBVP-transport-a}  using any RK  method satisfying the imaginary interval condition, \eqref{eqs:RK-IIC}. In particular,   the fully-discrete schemes based on the $s$-stage RK time discretization
\[
\bu_{n+1}=\Ps(\DbbL)\bu_n, \quad s=3,4, \qquad n=1,2,\ldots, 
\]
are stable under the CFL condition $\Delta t\cdot r_{{}_{\HN}}\!\!(\LN)\leq \CFL_s$,
\[
|\bu(t_n)|\leq 4(1+\sqrt{2})|\bu_0|.
\]
Observing  the simple bound, $\displaystyle r_{{}_{\HN}}\!\!(\LN)\leq\frac{a}{\Delta x}\const_{\bbH}$ with $\const_\bbH=2$, we end with CFL
condition sufficient for stability, $\mesh a \leq \CFL_s/2$.
\end{example}

 The last example depends on verifying weighted negativity, $\LN^\top\HN+\HN\LN\leq 0$, which requires the construction of a proper symmetrizer on a case by case basis.   A systematic approach for studying the weighted negativity for properly designed boundary  treatment augmenting centered difference schemes was developed in \cite{KS1974, Str1994, Gus1998, BEF2010}.
  To extend our RK stability framework  to  larger classes of difference approximations of initial-boundary values problems  requires   a more precise characterization of  the \emph{weighted} numerical range of Teoplitz-like spatial discretizations. This is left for future study.

\appendix

\section{The numerical range is $(1+\sqrt{2})$-spectral set}
In his remarkable work \cite{Cro2007}, Crouzeix proved that $W_H(A)$ is a $K$-numerical set with $K=11.08$ which was  later improved by Crouzeix \& Palencia to $K=1+\sqrt{2}$. We quote here the elegant proof of Ransford \&  Schwenninger \cite{RS2018} for Crouzeix \& Palencia $(1+\sqrt{2})$-bound, based on the following  lemma. In particular, we refer to the  recent review  \cite{SdV2023}. 
\begin{lemma}[{Ransford \&  Schwenninger $(1+\sqrt{2})$-spectral set}]
Let $T$ be a Hilbert space bounded  operator  $\|T\|< \infty$,   and let 
 $\Omega$ be a bounded
open set containing the spectrum of $T$. Suppose that for each  $f$ analytic on $\Omega$,  there exists an analytic $g$ on $\Omega$ such that
the following holds  $($here and below, $\|f\|_\Omega:= \sup_\Omega|f|$$)$:
\begin{equation}\label{eq:RS}
 \|g\|_{\Omega} \leq \|f\|_{\Omega} \ \ \textnormal{and} \ \ 
\|f(T) + g(T)^*\| \leq 2\|f\|_\Omega.
\ee
Then
\[
 \|f(T)\|\leq (1+\sqrt{2})\|f\|_\Omega
\]
\end{lemma}

\noindent
Proof. Let $\displaystyle K:=\sup_{\|f\|_\Omega=1}\|f(T)\|$. By assumption, for each $f, \ \|f\|_\Omega\leq 1$, there  exists $g$ such that \eqref{eq:RS} holds. Ransford \& Schwenninger invoked  the identity
\[
f(T)f(T)^*f(T)f(T)^* \equiv f(T)\big(f(T) + g(T)^*\big)^*f(T)f(T)^* - (fgf)(T)f(T)^*.
\]
A simple exercise shows that the norm of the quantity on the left equals
$\|f(T)\|^4$.
Since by \eqref{eq:RS}${}_1$, $\|(fgf)\|_\Omega \leq 1$ hence $\|fgf(T)\|\leq K$, and since by \eqref{eq:RS}${}_2$, $\|f(T) + g(T)^*\|\leq 2$, then  the expression on the right does not exceed
\[
\begin{split}
\|f(T)\|^4 & =\|f(T)f(T)^*f(T)f(T)^*\| \\
 & \leq \|f(T)\|\|f(T) + g(T)^*\| \|f(T)\|\|f(T)^*\| + \|(fgf)(T)\|\|f(T)^*\| \leq 2K^3+ K^2.
 \end{split}
\]
Hence, $\displaystyle K^4=\sup_{\|f\|_\Omega=1}\|f(T)\|^4\leq 2K^3+ K^2$
which implies $K\leq 1+\sqrt{2}$.\hfill $\square$\newline
Note that the lemma does not involve the numerical range of $T$ --- this comes into play in the construction  of $g=g_{{}_\Omega}$ satisfying \eqref{eq:RS}, in terms of Cauchy transform,
\[
g_{{}_\Omega}(z):=\frac{1}{2\pi i}\int_{\partial \Omega}\frac{\overline {f(\zeta)}}{\zeta-z}\rd{\zeta}, \quad z\in \Omega.
\]
The main thrust of the work, originated in \cite{vN1951} and  then developed in
\cite{Del1999} \cite{Cro2007} and finally \cite{CP2017}, is 
 to show that such $g_{{}_\Omega}$ with $\Omega=W_H(T)$ satisfies \eqref{eq:RS}. 

\bibliographystyle{plain}

\end{document}